\documentclass[11pt, a4paper, english]{smfart}
\usepackage{smfthm}
\usepackage{amssymb}
\usepackage{tikz-cd}
\usepackage{mathtools}
\usepackage{hyperref}
\usepackage{enumitem}

\hypersetup{
	pdftitle={The Hodge and Tate Conjectures for Hyper-Kähler Sixfolds of Generalized Kummer Type},
	pdfauthor={Salvatore Floccari},
	hidelinks
}

\renewcommand{\geq}{\geqslant}

\theoremstyle{plain}
\newtheorem{theorem}{Theorem}[section]

\newtheorem{definition}[theorem]{Definition}
\newtheorem{lemma}[theorem]{Lemma}

\newtheorem{proposition}[theorem]{Proposition}

\theoremstyle{remark}
\newtheorem{remark}[theorem]{Remark}

\renewcommand{\AA}{\ensuremath{\mathbb{A}}}
\newcommand{\CC}{\ensuremath{\mathbb{C}}}

\newcommand{\PP}{\ensuremath{\mathbb{P}}}
\newcommand{\QQ}{\ensuremath{\mathbb{Q}}}

\newcommand{\ZZ}{\ensuremath{\mathbb{Z}}}

\DeclareMathOperator{\Aut}{Aut}

\DeclareMathOperator{\NS}{\mathrm{NS}} 
\DeclareMathOperator{\Km}{\mathrm{Km}}

\title[{The Hodge and Tate conjectures for sixfolds of generalized Kummer type}]{The Hodge and Tate conjectures for hyper-K\"ahler sixfolds of generalized Kummer type}
\author{Salvatore Floccari}
\address{Institute of Algebraic Geometry, Leibniz~University~Hannover,~Germany}
\address{\textit{Current affiliation:} Humboldt-Universit\"at zu Berlin, Germany}
\email{salvatore.floccari@hu-berlin.de}

\begin{document}
		\begin{abstract}
			We prove the conjectures of Hodge and Tate for any six-dimensional hyper-K\"ahler variety that is deformation equivalent to a generalized Kummer variety. 
		\end{abstract}
	\maketitle

	\section{Introduction}
	
	The Hodge conjecture is one of the central open problems in complex algebraic geometry. It predicts that every rational Hodge class on a smooth projective variety is a linear combination of fundamental classes of its algebraic subvarieties. We refer to \cite{deligne2006hodge} for more information. 
	Known results are discussed in the book \cite{lewis}. For instance, it has been proven for rationally connected fourfolds \cite{blochSrinivas}, some Fermat hypersurfaces \cite{Shioda1979}, and in many degrees for unitary Shimura varieties \cite{bergeron2016}.
	The Hodge conjecture has been extensively studied for abelian varieties, see \cite{vanGeemen}. It was proven in  \cite{Tan83} for all simple abelian varieties of prime dimension, but remained
	open for a long time for abelian varieties of dimension~$4$. In a major development, after this paper was written, 
	Markman \cite{Markman2025Cycles} completed the proof of the Hodge conjecture for all abelian varieties of dimension at most $5$. 
	
	A hyper-K\"ahler manifold \cite{beauville1983varietes, Huy99} is a simply connected compact K\"ahler manifold whose $H^{2,0}$ is spanned by the class of a symplectic form.
	Besides K3 surfaces, the currently known examples fall into the deformation types commonly referred to as~$\mathrm{K}3^{[n]}$, $\mathrm{Kum}^n$,~$\mathrm{OG}10$, and~$\mathrm{OG}6$. The first two series, introduced by Beauville~\cite{beauville1983varietes}, 
	provide~$2n$-dimensional examples for every~$n\geq 2$, while the other two are O'Grady's deformation types~\cite{O'G99, O'G03}, in dimension $10$ and $6$, respectively. 
	The Hodge conjecture has been proven for several explicit examples of hyper-K\"ahler varieties, for example by \cite{deCataldoMigliorini2002, dCM04, FFZ,floccari22}; however, the absence of a geometric description of general members of the known families makes the study of algebraic cycles particularly difficult. 
	Remarkable progress has been obtained by Markman \cite{MarkmanBeauvilleBogomolov, markman2019monodromy, markmanRational}, who used Verbitsky's hyperholomorphic sheaves~\cite{verbitsky1997} to deform certain algebraic cycles on hyper-K\"ahler varieties; 
	related articles are \cite{CM13, voisinfootnotes, foster}. 
	
	Our main result is the following.
	\begin{theorem}\label{thm:HCKum3}
		Let $K$ be a projective manifold of $\mathrm{Kum}^3$-type. Then the Hodge conjecture holds for $K$, i.e.,~any cohomology class in $H^{i,i}(K)\, \cap\, H^{2i}(K,\QQ)$ is a $\QQ$-linear combination of fundamental classes of subvarieties of $K$, for any $i$.
	\end{theorem}
	
	This theorem provides the first complete deformation type of hyper-K\"ahler varieties of dimension at least $4$ for which the Hodge conjecture is known. As we explain below, a crucial input for the proof is our construction from \cite{floccariKum3} of a $\mathrm{K}3^{[3]}$-variety naturally associated with any variety of $\mathrm{Kum}^3$-type. 
	Some of the ideas developed here were subsequently applied together with Varesco \cite{floccariVaresco} to prove the Hodge conjecture for all fourfolds of $\mathrm{Kum}^2$-type. 
	%
	%
	%

 
	
	The Hodge conjecture is paralleled in the arithmetic setting by the Tate conjecture. Let $X$ be a smooth and projective variety over a finitely generated field~$k$, and let $\ell$ be a prime number. The absolute Galois group of~$k$ acts on the $\ell$-adic cohomology of $X_{\bar{k}}$. The strong form of the Tate conjecture predicts that this action is semisimple and that the Galois invariants in $H^{2i}_{\text{\'et}}(X_{\bar{k}}, \QQ_{\ell}(i))$ are $\QQ_{\ell}$-linear combinations of fundamental classes of $k$-subvarieties of~$X$ of codimension $i$, 
	see \cite{totaro}.
	The Tate conjecture is open already for divisors, in general. In characteristic $0$, it is known for divisors on abelian varieties~\cite{faltings}, hyper-K\"ahler varieties~\cite{Andre1996}, and many varieties with $h^{2,0}=1$ \cite{moonen2017}. 
	 
	\begin{theorem}\label{thm:TCKum3}
		Let $k\subset \CC$ be a finitely generated field, and assume that $K$ is a smooth and projective variety over $k$ whose base-change $K_{\CC}$ is a hyper-K\"ahler variety of $\mathrm{Kum}^3$-type. Then the strong Tate conjecture holds for $K$, for any prime number $\ell$. 
	\end{theorem}

	Theorem \ref{thm:TCKum3} follows from Theorem \ref{thm:HCKum3}, because the conjectures of Hodge and Tate are equivalent for any hyper-K\"ahler variety of known deformation type, by \cite{floccari2019, soldatenkov19, FFZ}.	
	
	\subsection*{Outline}
	The manifolds of $\mathrm{Kum}^n$-type are by definition deformation equivalent to Beauville's generalized Kummer variety $K^n(A)$ on an abelian surface~$A$ (\cite{beauville1983varietes}). 
	The Hodge conjecture holds for $K^n(A)$ by the results of de Cataldo-Migliorini~\cite{dCM04} and because it holds for any power of the abelian surface.	
	However, a deformation of a generalized Kummer variety is not necessarily of this form, and it is not clear how to approach the problem.
	
	In dimension $6$, the key ingredient needed to control algebraic cycles on any projective sixfold $K$ of $\mathrm{Kum}^3$-type is the construction which we gave in \cite{floccariKum3} of a K3 surface $S_K$ naturally associated with $K$. Geometrically, $K$ and $S_K$ are related by the existence of a projective variety $Y_K$ of $\mathrm{K}3^{[3]}$-type, which is at the same time a moduli space of sheaves on $S_K$ and a crepant resolution of the quotient $K/G$, where $G\cong (\ZZ/2\ZZ)^5$ is a group of involutions of $K$ acting trivially on its second cohomology. The construction relies on the deformation invariance of this group, proven by Hassett and Tschinkel \cite{hassettTschinkel}.
	
	Thanks to this construction we combined the work of O'Grady~\cite{O'G21}, Markman \cite{markman2019monodromy} and Voisin~\cite{voisinfootnotes} on the Kuga--Satake correspondence for $\mathrm{Kum}^n$-varieties with that of Varesco~\cite{varesco}, in order to prove that the Hodge conjecture holds for all powers of the K3 surfaces~$S_K$ arising as above (\cite[Corollary~5.8]{floccariKum3}). This result implies that the Hodge conjecture holds for the moduli space of sheaves $Y_K$ as well.
	
	Exploiting the relation between $K$ and $Y_K$ given by our construction from \cite{floccariKum3}, we deduce that any $G$-invariant rational Hodge class on $K$ is algebraic. Although this is not sufficient to prove Theorem \ref{thm:HCKum3}, via the LLV-decomposition (\cite{verbitsky1996cohomology,looijenga1997lie,green2019llv}) of the cohomology of $K$ it allows us to reduce its proof to the algebraicity of Hodge classes with the special property of remaining Hodge under every deformation of $K$. We call these canonical Hodge classes (Definition~\ref{def:canonical}); examples are given by the Chern classes of $K$. On a $\mathrm{Kum}^3$-manifold there are, however, many more canonical Hodge classes: besides those in $H^0(K,\QQ)$ and $H^{12}(K,\QQ)$, there are $17$-dimensional spaces of canonical Hodge classes in $H^4(K,\QQ)$ and $H^8(K,\QQ)$, and a $241$-dimensional space in $H^6(K,\QQ)$. 
	
	We prove in Theorem \ref{thm:HCcanonical} that canonical Hodge classes are analytic on any manifold of $\mathrm{Kum}^3$-type, and hence they are algebraic in the projective case. In degree $4$, inspired by Hassett and Tschinkel \cite{hassettTschinkel}, we show in Proposition~\ref{prop:canonicalH4} that they are all obtained as linear combinations of the second Chern class and the fundamental classes of certain natural fourfolds arising as components of fixed loci of automorphisms in $G$. These fourfolds deform along every deformation of $K$; they are absolutely trianalytic submanifolds in the sense of Verbitsky \cite{verbitsky1995trianalytic}. In degree $6$, a more involved construction is required to produce the necessary algebraic classes. In this case, we show that all canonical Hodge classes are obtained as linear combinations of pushforwards of divisor classes on the above fourfolds. 
	
	\subsection*{Notation and conventions}
	We work over the field of the complex numbers. A family of complex manifolds or algebraic varieties will be a smooth (submersive) and proper morphism $\mathcal{X}\to B$ over a smooth and connected base $B$. Tate twists in Hodge theory will be ignored throughout this text.
	If $X$ is a compact complex manifold, we say that a cohomology class $\alpha\in H^{\bullet}(X, \QQ)$ is analytic if it belongs to the subalgebra of $H^{\bullet}(X,\QQ)$ generated by Chern classes of coherent sheaves on $X$; if~$X$ is projective, then analytic classes are the same as algebraic ones, see \cite{voisinHodge}. Voisin \cite{Voisin2002Counterexample} has shown that there exist (non-projective) K\"ahler manifolds with Hodge classes which are not analytic.
	
	\subsection*{Acknowledgements}
	It is a pleasure to thank Stefan Schreieder for his detailed comments and suggestions which improved the present text. I am grateful to Lie Fu for many stimulating conversations. I thank Domenico Valloni for his encouragement and support through countless discussions.

	\section{Cohomology of $\mathrm{Kum}^3$-manifolds}\label{sec:2}
	
	In this section we recall several known properties of the cohomology of hyper-K\"ahler manifolds (of dimension $>2$). We then discuss in detail the case of $\mathrm{Kum}^3$-type. As general references about hyper-K\"ahler manifolds we recommend \cite{beauville1983varietes} and \cite{Huy99}.
	
	\begin{definition}\label{def:canonical}
		Let $X$ be a hyper-K\"{a}hler manifold. A canonical Hodge class on $X$ is a cohomology class which stays Hodge on every deformation of $X$, i.e., whose parallel transport is a Hodge class for all parallel transport operators arising from any smooth holomorphic family of hyper-K\"ahler manifolds with $X$ as one special fiber.
	\end{definition}
	
	Any polynomial expression in the Chern classes of $X$ gives a canonical Hodge class. Another example comes from the Beauville-Bogomolov form $q_X$ on $H^2(X,\QQ)$: identifying $H^2(X,\QQ)$ with its dual by means of this form, the pairing determines a class~$\overline{q}_X$ in~$\mathrm{Sym}^2(H^2(X,\QQ))$; the latter space embeds in~$H^4(X,\QQ)$ via cup-product by a theorem of Verbitsky \cite{verbitsky1996cohomology}, and we obtain a canonical Hodge class $\overline{q}=\overline{q}_X \in H^4(X,\QQ)$. 
	\begin{remark} \label{rmk:BBclass}
		If $a_1, \dots , a_m$ is an orthogonal basis of $H^2(X,\QQ)$, we have 
		\[
		\overline{q}_X = \sum_{i=1}^m \frac{ a_i^2}{q_X(a_i,a_i)} \in H^4(X,\QQ).
		\] 
	\end{remark}

	The following result of Fujiki \cite{fujiki1987} is very useful: if $\omega\in  H^{4k}(X,\QQ)$ is a canonical Hodge class, there exists $C(\omega)\in \QQ$ (the generalized Fujiki constant of $\omega$) such that
	\begin{equation}\label{eq:fujiki}
		\int_X \omega\cdot \gamma^{\dim X - 2k} = C(\omega) \cdot q_X(\gamma, \gamma)^{\tfrac{1}{2}\dim X - k}, 
	\end{equation}
	for any $\gamma\in H^2(X,\QQ)$. The Fujiki constant of $X$ is by definition $C(1)$.
	
	Another important tool is the action of the LLV-Lie algebra $\mathfrak{g}(X)$ on the cohomology of~$X$, studied first by Verbitsky \cite{verbitsky1996cohomology} and Looijenga-Lunts \cite{looijenga1997lie}. For $\gamma \in H^2(X,\QQ)$, consider the endomorphism $e_{\gamma}$ of $H^{\bullet}(X,\QQ)$ of degree $2$ defined by $e_\gamma(\beta)=\gamma \cdot \beta$.
	It turns out that, whenever $q_X(\gamma , \gamma)\neq 0$, there exists an adjoint endomorphism $\lambda_\gamma $ of $H^{\bullet}(X,\QQ)$ of degree~$-2$ such that $e_\gamma, \theta, \lambda_\gamma $ generate a copy of $\mathfrak{sl}_2$ inside $\mathfrak{gl}(H^{\bullet}(X, \QQ))$, where the action of $\theta$ on each~$H^{k}(X,\QQ)$ is multiplication by $k-\dim X$. 	
	Then $\mathfrak{g}(X)$ is defined as the Lie subalgebra of $\mathfrak{gl}(H^{\bullet}(X,\QQ))$ generated by all $\mathfrak{sl}_2$-triples as above. By definition, this Lie algebra only depends on the topology of $X$.
		
	For any hyper-K\"ahler manifold $X$ we have 
	\[\mathfrak{g}(X)=\mathfrak{so}(H^2(X,\QQ)\oplus \mathrm{U}),\] 
	where $\mathrm{U}$ denotes a hyperbolic plane; the lattice $H^2(X,\ZZ)\oplus \mathrm{U}$ is called the Mukai lattice of $X$. 
	The centralizer $\mathfrak{g}^0(X)$ of $\theta$ is isomorphic to $\QQ \cdot \theta\, \oplus \,\mathfrak{so}(H^2(X,\QQ))$ and it acts on the cohomology preserving the degree. Moreover, the LLV-Lie algebra is compatible with Hodge theory, meaning that any $\mathfrak{so}(H^2(X,\QQ))$-submodule of the cohomology of $X$ is automatically a Hodge substructure; 
	the canonical Hodge classes on $X$ correspond to the trivial subrepresentations of $\mathfrak{so}(H^2(X,\QQ))$ on the cohomology. They form a subalgebra, because the LLV-action of $\mathfrak{so}(H^2(X,\QQ))$ on the cohomology algebra is via derivations.
	
	To close this summary, we introduce the group 
	$$\Aut_0(X)\coloneqq \{ f\in \Aut(X) \ | \ f^*_{|_{H^2(X,\ZZ)}} = \mathrm{id} \}$$ of automorphisms of $X$ acting trivially on its second cohomology. It is a finite group (\cite[Proposition 9.1]{Huy99}) whose action on the cohomology commutes with the LLV-Lie algebra. The key property of $\Aut_0(X)$ is that it is deformation-invariant, by \cite[Theorem~2.1]{hassettTschinkel}.
	
	\subsection{The generalized Kummer sixfold}
	By definition, $\mathrm{Kum}^3$-manifolds are deformations of the generalized Kummer sixfolds $K^3(A)$, where $A$ is an abelian surface or a $2$-dimensional complex torus. These are constructed as follows in \cite[\S7]{beauville1983varietes}. Let  $A^{[4]}$ be the Douady space of length $4$ subspaces of dimension $0$ on $A$ (if $A$ is projective,~$A^{[4]}$ is the Hilbert scheme of $0$-dimensional subschemes of length $4$ on $A$), which is a crepant resolution of the fourth symmetric product $A^{(4)}$ of $A$. The resolution map $\nu\colon A^{[4]}\to A^{(4)}$, commonly called the Hilbert-Chow morphism, sends~$\xi\in A^{[4]}$ to its support. Consider the summation map $\Sigma\colon A^{(4)}\to A$, which sends $(a_1,a_2,a_3,a_4)\in A^{(4)}$ to $\sum_i a_i \in A$. Then the composition $\Sigma\circ \nu\colon A^{[4]}\to A$ is an isotrivial fibration and 
	$$K^3(A) \coloneqq (\Sigma\circ \nu)^{-1}(0) $$
	is a hyper-K\"ahler manifold, which is a crepant resolution of $A_0^{(4)}\coloneqq \Sigma^{-1}(0)$.
	
	The Hodge numbers of $\mathrm{Kum}^3$-manifolds have been determined in \cite{GS1993}:
	\begin{equation}\label{eq:hodgeNumbers}
	\begin{tabular}{c c c c c c c c c c c c c}
		 &  &  &  &  & & 1 &    &  & &  &  & \\
		&  &  &   &  & 0  &  & 0  & &  & &   & \\
	 &  &  &  & 1 & & 5  &  & 1 & &  &  & \\
	&  &  & 0 &  & 4  &  & 4  & & 0 & &   & \\
	 &  & 1 &  & 6 & & 37  &  & 6 & & 1 &  & \\
	 & 0 &  & 4 &  & 24  &  & 24  & & 4 & & 0  & \\
	1 &  & 5 &  & 37 & & 372  &  & 37 & & 5 &  & 1
	\end{tabular} 
\end{equation}
	
	\subsection{The LLV-decomposition}\label{subsec:LLV}
	Denoting by $V$ the (rational) Mukai lattice of a $\mathrm{Kum}^3$-manifold $K$, the decomposition of the cohomology into isotypical $\mathfrak{so}(V)$-components is given by Green-Kim-Laza-Robles in \cite[Corollary 3.6]{green2019llv}:
	\begin{equation}\label{eq:LLV}
	H^{\bullet}(K, \QQ) = V_{(3)} \oplus V_{(1,1)} \oplus 16V\oplus 240\QQ \oplus V_{\left(\tfrac{3}{2}, \tfrac{1}{2}, \tfrac{1}{2}, \tfrac{1}{2}\right)} ,
	\end{equation}
	where $V_{(3)}$ is the subalgebra generated by the second cohomology, the summand $V_{(1,1)}$ is isomorphic to~$\bigwedge^2 V$, and the last summand is the odd cohomology, which is an irreducible spinor representation. This is in fact a decomposition into submodules for the action of $H^2(K,\QQ)$ via cup-product. The components have ranks:
	\begin{equation}\label{eq:ranks}
		\begin{tabular}{|c|c|c|c|c|}
			\hline
			$V_{(3)}$ & $V_{(1,1)}$ & $16V$ & $240\QQ$ & $V_{\left(\tfrac{3}{2}, \tfrac{1}{2}, \tfrac{1}{2}, \tfrac{1}{2}\right)}$ \\
			\hline 
			$156$ & $36$ & $144$ & $240$ & $128$ \\
			\hline
		\end{tabular}
	\end{equation}
	
	Denoting by $\bar{V}$ the summand $H^2(K, \QQ)$ of~$V$, the $\mathfrak{so}(\bar{V})$-representation on the cohomology is given by (we only display the cohomology in even degrees)
	\begin{equation}\label{eq:LLVrefined}
		\begin{tabular}{|c|c|c|c|c|c|}
			\hline
			$\mathrm{deg}$ & $V_{(3)}$ & $V_{(1,1)}$ & $16V$ & $240\QQ$ & $\mathrm{rk}$\\
			\hline 
			0 & $\QQ$ & 0 & 0 & 0 & 1 \\ \hline
			2 & $\bar{V}$ & 0 & 0 & 0 & 7 \\ \hline
			4 & $\mathrm{Sym}^2(\bar{V}) $ & $\bar{V}$ & $16\QQ$ & 0 & $51$ \\ \hline
			6 & $\mathrm{Sym}^3(\bar{V}) $ & $\bigwedge^2 \bar{V} \oplus \QQ$ & $16\bar{V}$ & $240\QQ$ &  $458$ \\ \hline
			8 & $\mathrm{Sym}^2(\bar{V}) $ & $\bar{V}$ & $16\QQ$ & 0 & 51 \\ \hline
			10 & $\bar{V}$ & 0 & 0 & 0 & 7 \\ \hline
			12 & $\QQ$ & 0 & 0 & 0 & 1 \\ \hline	
		\end{tabular}
	\end{equation}
	This also describes the Hodge structure on $H^{2\bullet}(K,\QQ)$ in terms of that on $H^2(K,\QQ)$.
	
	\subsection{Generalized Fujiki constants}\label{subsec:FujikiConst}
	The following table shows the generalized Fujiki constants for the monomials in $\overline{q}$ and the Chern classes on $\mathrm{Kum}^3$-manifolds. 
	\begin{equation}\label{eq:constants}
		\begin{tabular}{|c|c|c|c|}
			\hline
			$C(1) = 60$ & $C(\overline{q}) =132$ & $C(\overline{q}^2)=396$ & $C(\overline{q}^3)=2772$ \\
			\hline
			$C(c_2)= 288 $ & $C(\overline{q} \cdot c_2)= 864$ & $C(\overline{q}^2 \cdot c_2)= 6048$ & --- \\
			\hline 
			$ C(c_2^2)=1920$ & $C(\overline{q} \cdot c_2^2)=13440$ & $C(c_4)=480$ & $C(\overline{q} \cdot c_4)=3360$  \\
			\hline 
			$C(c_2^3)=30208$ & --- & $C(c_2c_4)=6784$ & $C(c_6)=448$ \\
			\hline
		\end{tabular}
	\end{equation}

	The constant $C(1)$ can be found in \cite{rapagnetta2008beauville}, while the generalized Fujiki constants for monomials in the Chern classes of degree at most $4$ and their products with powers of $\overline{q}$ can be calculated via the results of Beckmann and Song \cite[Corollary 2.7, Proposition 2.4, Example 2.13]{BS22}.
	The Chern numbers have been determined by Nieper-Wisskirchen in \cite{nieper}.
	
	\subsection{Canonical Hodge classes} \label{subsec:HodgeCanonical}
	There are however many more canonical Hodge classes on $\mathrm{Kum}^3$-manifolds than those which are polynomials in $\overline{q}$ and the Chern classes. 
	By the LLV-decomposition \eqref{eq:LLVrefined}, the space of canonical Hodge classes in degree $4$ has dimension $17$, spanned by $\overline{q}$ and the copy of $16\QQ$ coming from the third summand in \eqref{eq:LLV}, while canonical Hodge classes in the middle cohomology form a vector space of dimension~$241$, spanned by the unique (up to multiples) canonical Hodge class in the LLV-submodule~$V_{(1,1)}$ and the summand $240\QQ$ of~\eqref{eq:LLV}. Canonical Hodge classes in degree $8$ form another $17$-dimensional subspace.
	
	\subsection{Automorphisms trivial on the second cohomology}\label{subsec:aut0}
	Let $A_k$ denote the subgroup of $k$-torsion points of a $2$-dimensional complex torus $A$. By \cite{boissiere2011higher}, we have 
	\[
	\Aut_0(K^3(A)) = A_4\rtimes \langle -1\rangle,
	\]
	where the second factor acts on the first as the inverse. The action on $K^3(A)$ is induced by the natural action on $A^{[4]}$; by \cite{oguiso2020no}, the involution $-1$ acts as multiplication by $-1$ on the third cohomology.
	There are two subgroups which will play an important role.
	\begin{itemize}
		\item Let $\Gamma\coloneqq A_4\subset \Aut_0(K^3(A))$; it is the subgroup of automorphisms which act trivially on $H^3(K^3(A),\ZZ)$ (and hence on the whole odd cohomology, since it is an irreducible LLV-module). We have $\Gamma\cong(\ZZ/4\ZZ)^4$.
		\item Let $G\coloneqq A_2\times \langle -1\rangle \subset \Aut_0(K^3(A))$; note that $G\cong (\ZZ/2\ZZ)^5$ is abelian and intersects~$\Gamma$ in $A_2$.
	\end{itemize} 
   	Thanks to the deformation-invariance of the automorphisms trivial on the second cohomology, for any $K$ of $\mathrm{Kum}^3$-type we have $\Aut_0(K) \cong (\ZZ/4\ZZ)^4 \rtimes \ZZ/2\ZZ$, and the subgroups $G\cong (\ZZ/2\ZZ)^5$ and $\Gamma\cong (\ZZ/4\ZZ)^4$.
   	
   	\begin{remark}\label{rmk:definingG}
   		For an arbitrary $K$ of $\mathrm{Kum}^3$-type, the group $G$ is canonically defined as the subgroup of $\Aut_0(K)$ generated by the automorphisms $g$ such that the fixed locus~$K^g$ has a $4$-dimensional component, see \cite[\S2]{floccariKum3}.
   		We will eventually prove that $G\subset \Aut_0(K)$ is the subgroup of automorphisms acting trivially on $H^4(K,\ZZ)$, see Remark \ref{rmk:characterizationG}.
   	\end{remark}

	\begin{lemma}\label{lem:Gammainvariants}
	For any $K$ of $\mathrm{Kum}^3$-type, the space of $\Gamma$-invariants in $H^{\bullet}(K,\QQ)$ is a LLV-submodule
	\[ 
	(H^{\bullet}(K,\QQ))^{\Gamma} \cong V_{(3)} \oplus V_{(1,1)} \oplus V \oplus V_{\left(\tfrac{3}{2}, \tfrac{1}{2}, \tfrac{1}{2}, \tfrac{1}{2}\right)} .
	\]
	Moreover:	
		\begin{enumerate}[label=(\roman*)]
			\item the space of canonical Hodge classes in $H^4(K, \QQ)^{\Gamma}$ is two-dimensional, spanned by the linearly independent classes $c_2$ and $\overline{q}_{K}$;		
			\item the space of canonical Hodge classes in $H^6(K,\QQ)^{\Gamma}$ is one-dimensional, and it is contained in the LLV-summand $V_{(1,1)}$.
		\end{enumerate}
	\end{lemma}
	\begin{proof}
		It suffices to work with the generalized Kummer variety $K^3(A)$ on an abelian surface $A$.
		Consider the LLV-decomposition \eqref{eq:LLV}:
		\[
		H^{\bullet}(K^3(A),\QQ) = V_{(3)}\oplus V_{(1,1)} \oplus 16V\oplus 240\QQ \oplus V_{\left(\tfrac{3}{2}, \tfrac{1}{2}, \tfrac{1}{2}, \tfrac{1}{2}\right)}.
		\]
		The group $\Gamma$ acts trivially on $V_{(3)}$ and $V_{\left(\tfrac{3}{2}, \tfrac{1}{2}, \tfrac{1}{2}, \tfrac{1}{2}\right)}$.
		In degree $4$, the decomposition becomes
		\begin{equation} \label{eq:LLVdeg4}
			H^4(K^3(A), \QQ) = \mathrm{Sym}^2(H^2(K^3(A), \QQ))\oplus H^2(K^3(A),\QQ) \oplus 16\QQ,
		\end{equation} and $\Gamma$ acts trivially on the first summand. 		
		By \cite[p. 769]{beauville1983varietes}, the quotient of $K^3(A)\times A$ by the diagonal action of $\Gamma=A_4$ is $A^{[4]}$. 
		Since the group $\Gamma$ acts trivially on the cohomology of ${A}$ as well as on the cohomology groups~$H^2(K^3(A),\QQ)$ and~$H^3(K^3(A),\QQ)$, we have
		\begin{align*}
			H^4 (A^{[4]},\QQ) = & H^4(K^3(A),\QQ)^{\Gamma}\oplus \bigl(H^3(K^3(A),\QQ)\otimes H^1(A,\QQ)\bigr) \\ & \oplus\bigl( H^2(K^3(A),\QQ)\otimes H^2(A,\QQ)\bigr) \oplus H^4(A,\QQ).
		\end{align*}
	By \cite{GS1993}, the Hodge numbers of the left hand side are $$h^{4,0}(A^{[4]})=2, \  \ \  h^{3,1}(A^{[4]})=23,\ \ \ h^{2,2}(A^{[4]})=61.$$
	Calculating the Hodge numbers of the last $3$ summands on the right-hand side with \eqref{eq:hodgeNumbers}, we deduce that 
		\[
		H^4(K^3(A),\QQ)^{\Gamma} = \mathrm{Sym}^2(H^2(K^3(A),\QQ)) \oplus T
		\]
		for an effective Hodge structure $T$ of weight $4$ with Hodge numbers $(0,1,6,1,0)$.
		Since~$\Gamma$ respects the LLV-decomposition \eqref{eq:LLVdeg4}, we find that $T\cong H^2(K^3(A),\QQ)\oplus \QQ$. This implies that the summand $V_{(1,1)}$ of \eqref{eq:LLV} is contained in $ H^{\bullet}(K^3(A),\QQ)^{\Gamma}$, and that $(16V)^{\Gamma}\cong V$.
		
		Next, the results of \cite{GS1993} give $b_6(A^{[4]})= 592$ and we write
		\begin{align*}
			H^6 (A^{[4]}, \QQ) = H^6(K^3(A), \QQ)^{\Gamma} \oplus R
		\end{align*}
		where \begin{align*}
			R=\bigl(H^5(K^3&(A), \QQ)  \otimes H^1(A, \QQ)\bigr) \oplus \bigl( H^4(K^3(A), \QQ)^{\Gamma} \otimes H^2(A, \QQ)\bigr) \\ & \oplus \bigl(H^3(K^3(A), \QQ)\otimes H^3(A, \QQ)\big) \oplus H^2(K^3(A),\QQ).
		\end{align*}
		We have just determined the invariants for the action of $\Gamma$ on $H^4(K^3(A),\QQ)$, and we calculate that $\mathrm{dim}_{\QQ}(R)= 479$, so that $H^6(K^3(A),\QQ)^{\Gamma}$ has rank $113$.
		Table \eqref{eq:LLVrefined} now gives
		$$H^6(K^3(A),\QQ)^{\Gamma} \cong \mathrm{Sym}^3(H^2(K^3(A),\QQ)) \oplus \left( \bigwedge^2 H^2(K^3(A), \QQ) \oplus \QQ\right)\oplus H^2(K^3(A),\QQ) \oplus S, $$
		where $S$ is the space of $\Gamma$-invariants on the LLV-trivial summand $240\QQ$.
		But then 
		$$ \mathrm{dim}_{\QQ}(S)=113 - 84 - 21 - 1 - 7 = 0.$$
		We have thus shown that 
		\[ 
		(H^{\bullet}(K^3(A),\QQ))^{\Gamma} \cong V_{(3)} \oplus V_{(1,1)} \oplus V \oplus V_{\left(\tfrac{3}{2}, \tfrac{1}{2}, \tfrac{1}{2}, \tfrac{1}{2}\right)}. 
		\]
		
		Therefore, the only (up to multiples) $\Gamma$-invariant canonical Hodge class in $H^6(K^3(A),\QQ)$ is the one which comes from the summand $V_{(1,1)}$ in the LLV-decomposition, which proves~$(ii)$.
		The canonical Hodge classes in $\mathrm{Sym}^2(H^2(K^3(A),\QQ))$ are all multiples of~$\overline{q}_{K^3(A)}$. It follows that the $\Gamma$-invariant canonical Hodge classes in $H^4(K^3(A),\QQ)$ form a $2$-dimensional vector space, which clearly contains $\overline{q}_{K^3(A)} $ and $c_2$. These classes are linearly independent by \cite[Example 2.13]{BS22}, and $(i)$ follows.
	\end{proof}

	\subsection{A basis for $\Gamma$-invariant canonical Hodge classes}
	It will be convenient for our computations to define
		\begin{equation}\label{eq:z}
			 z\coloneqq c_2 - \tfrac{C(c_2)}{C(\overline{q})}\overline{q} = c_2-\tfrac{24}{11}\overline{q} \in H^4(K,\QQ).
		 \end{equation}
	The point is that $z$ is a canonical Hodge class and $C(z)=0$; by \cite[Proposition 2.4]{BS22}, it follows that $C(\overline{q}\cdot z)=0$ and $\overline{q}^2\cdot z=0$. The classes $\overline{q}$ and $z$ form a basis of $\Gamma$-invariant canonical Hodge classes in degree $4$.
	
	\begin{lemma}\label{lem:basis}
		The classes $\overline{q}^2$ and $\overline{q}\cdot z$ are linearly independent and give a basis of $\Gamma$-invariant canonical Hodge classes in $H^{8}(K, \QQ)$. We have:
		\begin{equation*}
			\begin{tabular}{|c|c|c|}
				\hline 
				$C(z^2) =\tfrac{384}{11} $ & $C (\overline{q} \cdot z^2) = \tfrac{2688}{11} $ & $ C(z^3) = -\tfrac{22016}{121} $ \\
				\hline
				$z^2= \tfrac{32}{363} \overline{q}^2 - \tfrac{172}{231} \overline{q} \cdot z$ \rule{0pt}{2.4ex}& $c_2^2 = \tfrac{160}{33} \overline{q}^2 + \tfrac{76}{21}\overline{q}\cdot z$ & $c_4= \tfrac{40}{33}\overline{q}^2 - \tfrac{47}{21}\overline{q}\cdot z$ \\
				\hline
			\end{tabular}
		\end{equation*}
	\end{lemma}
	\begin{proof}
		Using the constants of table \eqref{eq:constants}, we find the equations
		$$13440=\overline{q}\cdot c_2^2 = \overline{q} \cdot (\tfrac{24^2}{11^2} \overline{q}^2 + \tfrac{48}{11} \overline{q} \cdot z + z^2) = \tfrac{24^2}{11^2} \overline{q}^3 +\overline{q} \cdot z^2$$
		and
		\[30208 = c_2^3=\tfrac{24^3}{11^3} \overline{q}^3 + 3\cdot \tfrac{24}{11} \overline{q} \cdot z^2 + z^3,\]
		which give $\overline{q}\cdot z^2 = \tfrac{2688}{11}$ and $z^3=-\tfrac{22016}{121}$.
		These numbers imply that $\overline{q}\cdot z$ is not zero, and hence $\overline{q}^2 $, $\overline{q}\cdot z$ form a basis for $\Gamma$-invariant canonical Hodge classes in $H^8(K, \QQ)$. 
		By \cite[Proposition 2.4]{BS22}, we have $C(\overline{q}\cdot z^2)=7C(z^2)$ and we find $C(z^2)=\tfrac{384}{11}$. We have
		$$z^2 = \tfrac{C(z^2)}{C(\overline{q}^2)} \overline{q}^2 + \lambda \overline{q} \cdot z = \tfrac{32}{363} \overline{q}^2 + \lambda \overline{q} \cdot z,$$
		for some $\lambda\in \QQ$.
		Taking the cup-product with $z$ we find $z^3=\lambda \overline{q}\cdot z^2$, and so $\lambda=-\tfrac{172}{231}$.
		Similarly, writing $c_2^2=\tfrac{C(c_2^2)}{C(\overline{q}^2)}\overline{q}^2 + a \overline{q} \cdot z=\tfrac{160}{33}\overline{q}^2 + a\overline{q}\cdot z$ we obtain
		\[
		30208 = c_2^3 = \tfrac{24}{11}\cdot \tfrac{160}{33} \overline{q}^3 + a \overline{q} \cdot z^2,
		\]
		so that $a = \tfrac{76}{21}$, while writing $c_4= \tfrac{C(c_4)}{C(\overline{q}^2)}\overline{q}^2 + b \overline{q} \cdot z= \tfrac{40}{33}\overline{q}^2 + b \overline{q} \cdot z$ we obtain
		\[
		6784 = c_2c_4 = \tfrac{24}{11}\cdot \tfrac{480}{396} \overline{q}^3 + b \overline{q}\cdot z^2,
		\]
		from which we find $b = -\frac{47}{21}$.	
		\end{proof}
	
		\section{Geometry of generalized Kummer sixfolds} \label{sec:3}
		
	The following construction is of fundamental importance for the present article. 
	\begin{theorem}[\cite{floccariKum3}]\label{thm:associatedK3}
		Let $K$ be any $\mathrm{Kum}^3$-manifold. There exists a crepant resolution 
		$$Y_K\to K/G $$
		with $Y_K$ a manifold of $\mathrm{K}3^{[3]}$-type.
		Moreover, $W\coloneqq \bigcup_{1\neq g\in G} K^g$ is the union of $16$ transverse fourfolds of $\mathrm{K}3^{[2]}$-type, and, letting $\tilde{K}\coloneqq \mathrm{Bl}_W(K)$, we have $Y_K= \tilde{K}/G$.
		
		If $K$ is projective, then there exists a uniquely determined (up to isomorphism) K3 surface~$S_K$ such that $Y_K$ is birational to a moduli space of stable sheaves on $S_K$.
	\end{theorem}
	In the above statement, $G\cong (\ZZ/2\ZZ)^5$ is the subgroup of $\Aut_0(K)$ defined in \S\ref{subsec:aut0}.
	Taking moduli spaces of stable sheaves on K3 surfaces is a prominent source of constructions of hyper-K\"ahler varieties of $\mathrm{K}3^{[n]}$-type, \cite{huybrechts2010geometry}. By \cite{BM14a, BM14b}, in the projective case $Y_K$ can be realized as a moduli space of stable objects in the derived category of $S_K$.
	
	\begin{remark}\label{rmk:iteratedBlowUp}
		Choosing any ordering $W_i$, $i=1,\dots,16$, of the irreducible components of $W$, the blow-up $\tilde{K}\to K$ may be performed as the successive blow-up of the strict transforms of the $W_i$. This can be checked via a local computation or by applying the general~\cite[Proposition 2.10]{kiem}. 
		We thus find a sequence
		\[
		\tilde{K}=\tilde{K}_{16} \xrightarrow{p_{16}} \tilde{K}_{15} \xrightarrow{p_{15}}\dots \xrightarrow{p_2} \tilde{K}_1 \xrightarrow{p_1}\tilde{K}_0 = K,
		\]
		where $p_i$ is the blow-up of $\tilde{K}_{i-1}$ in a smooth submanifold $\bar{W}_{i} \subset \tilde{K}_{i-1}$; each $\bar{W}_i$ is the iterated blow-up of the $\mathrm{K}3^{[2]}$-manifold $W_i$ along smooth surfaces.
	\end{remark}
	
	The components $W_i$ of $W$ can be described very explicitly when $K=K^3(A)$ is the generalized Kummer sixfold on a complex torus $A$.
	
	\subsection{Notation}\label{subsec:someNotation}
	We denote by $A_k\subset A$ the subgroup of $k$-torsion points; we have $A_k\cong (\ZZ/k\ZZ)^4$. Given a point $\tau\in A_2$ we will denote by $A_{2,\tau}\subset A_4$ the subset of those $a$ such that $2a=\tau$ (equivalently, such that $a=-a+\tau$). Translation by any $\tau'\in A_2$ maps $A_{2,\tau}$ to itself, showing that $A_{2,\tau}$ is a torsor under $A_{2}$. If $\tau\in A_2$, by abuse of notation we will denote by $\tau$ (resp. by $(\tau, -1)$) the automorphism of $A$ given by $\tau(a)=a+\tau$ (resp. by $(\tau, -1)(a) = -a +\tau$). 
	For any $\tau\in A_{2}$, the quotient $A/\langle (\tau, -1) \rangle$ is a singular surface with $16$ double points corresponding to the points of $A_{2,\tau}$. Blowing-up these points gives a K3 surface $\Km^{\tau}(A)$, isomorphic to the Kummer surface associated to $A$.
	
	\begin{definition}\phantomsection\label{def:relevantSubvarieties}
		\begin{enumerate}[label=(\roman*)]
			\item For any $\tau\in A_2$, we let $W_{\tau}\subset K^3(A)$ be the strict transform of 
			\[
			\{(a, b, -a+\tau, -b+\tau), \ a,b\in A\} \subset A_{0}^{(4)}
			\]
			under the Hilbert-Chow morphism $\nu\colon K^3(A)\to A_0^{(4)}$.
			\item For any $\tau\neq \tau' \in A_2$, we let $V_{\tau, \tau'} \subset K^3(A)$ be the strict transform of 
			\[
			\{(a, a+\tau+\tau', -a+\tau, -a+\tau'), \ a \in A\} \subset A_{0}^{(4)}
			\]
			under the Hilbert-Chow morphism $\nu\colon K^3(A)\to A_0^{(4)}$. Note that $V_{\tau, \tau'} = V_{\tau', \tau}$.
		\end{enumerate}
	\end{definition}
	
	The $W_{\tau}$, $\tau\in A_2$, are the $16$ components of the locus $W=\bigcup_{1\neq g\in G} (K^3(A))^g \subset K^3(A)$ appearing in Theorem \ref{thm:associatedK3} in the special case $K=K^3(A)$. We summarize in the next two Propositions several results obtained in \cite[\S2]{floccariKum3} (where the notation is slightly different).
	\begin{proposition}\phantomsection \label{prop:relevantSubvarieties}
		\begin{enumerate}[label=(\roman*)]
			\item Let $\tau\in A_2$. Then $W_{\tau}$ is a hyper-K\"ahler manifold of $\mathrm{K}3^{[2]}$-type, isomorphic to $\Km^{\tau}(A)^{[2]}$. The Hilbert-Chow morphism $\nu\colon K^3(A)\to A^{(4)}$ restricts to a resolution 
			\[
			W_{\tau}=\Km^{\tau}(A)^{[2]} \xrightarrow{\ \nu\ } (A/\langle (\tau, -1) \rangle)^{(2)}.
			\]
			\item Let $\tau\neq \tau'\in A_2$. Then $V_{\tau, \tau'}$ is a $\mathrm{K}3$ surface isomorphic to $\Km^{\tau}(A/\langle \tau+\tau'\rangle)$, and $\nu$ restricts to a resolution
			\[
			V_{\tau, \tau'}=\Km^{\tau}(A/\langle \tau+\tau'\rangle) \xrightarrow{\ \nu \ } A/\langle (\tau, -1), (\tau', -1)\rangle. 
			\] 
			\item For $\tau$ varying in $A_2$ we obtain $16$ distinct codimension $2$ submanifolds $W_{\tau}$ of~$K^3(A)$. For any $\tau\neq \tau'$, the submanifolds $W_{\tau} $ and $W_{\tau'}$ intersect transversely in~$V_{\tau, \tau'}$. For~$\tau, \tau', \tau''$ pairwise distinct, $W_{\tau}, W_{\tau'}, W_{\tau''}$ meet transversely in $4$ points, while $4$ or more distinct components $W_{\tau}$ have empty intersection.
		\end{enumerate} 
	\end{proposition} 

	Consider the group $\Aut_0(K^3(A)) = A_4 \rtimes \langle -1 \rangle$, with its subgroups $G=A_2\times \langle -1 \rangle$ and $\Gamma= A_4$ as in \S\ref{subsec:aut0}. The submanifolds introduced above arise from the study of the fixed loci of these automorphisms.
	\begin{proposition} \phantomsection \label{prop:groupAction}
		\begin{enumerate}[label=(\roman*)]
			\item Any automorphism $g\in G$ restricts to an automorphism of each of the $W_{\tau}$ and of the $V_{\tau, \tau'}$. An automorphism $\gamma_{\epsilon} \in \Gamma$ induced by translation by $\epsilon \in A_4$ restricts to isomorphisms $$W_{\tau}\xrightarrow{\ \sim \ } W_{\tau+2\epsilon} \ \ \ \text{ and } \ \ \ V_{\tau, \tau' } \xrightarrow{\ \sim \ } V_{\tau+2\epsilon, \tau' +2\epsilon }. $$
			\item The manifold $W_{\tau}$ is the unique positive-dimensional component of the fixed locus of the automorphism $(\tau, -1)\in G$. The induced action of $G/ \langle (\tau, -1) \rangle \cong (\ZZ/2\ZZ)^4 $ on $W_{\tau}$ is identified with the natural action of $A_2$ on $\Km^{\tau}(A)^{[2]}=W_{\tau}$.
			\item For $\tau\neq 0\in A_2$, the fixed locus of the automorphism $\tau\in G$ is the disjoint union of the $8$ $\mathrm{K}3$ surfaces $V_{\tau -\theta, \theta}$, for $\theta\in A_2$. The stabilizer in $G$ of $V_{\tau, \tau'}$ is the subgroup~$\langle (\tau, -1), (\tau', -1) \rangle \cong (\ZZ/2\ZZ)^2$.
		\end{enumerate} 
	\end{proposition} 
	\begin{remark}\label{rmk:deformations}
		The submanifolds $\iota_\tau\colon W_{\tau}\to K^3(A)$ deform with $K^3(A)$.
		In fact, let~$\mathcal{K}\to B$ be a family of manifolds of $\mathrm{Kum}^3$-type with $ 0\in B$ such that $\mathcal{K}_0 = K^3(A)$ for some $2$-dimensional complex torus $A$. The groups $\Aut_0(\mathcal{K}_b)$ yield a local system over~$B$; up to a finite \'etale base-change, we have a well-defined fibrewise action of $\Aut_0(K^3(A))$ on~$\mathcal{K} \to B$. For each $\tau\in A_2$, a component of the fixed locus $\mathcal{K}^{(\tau,-1)}$ is a subfamily~$\mathcal{W}_{\tau}\subset \mathcal{K}$ of manifolds of $\mathrm{K}3^{[2]}$-type over $B$, with $\mathcal{W}_{\tau,0}=W_{\tau}$. 
		This implies that the pull-back map $\iota_{\tau}^*\colon H^2(K^3(A), \ZZ)\to H^2(W_{\tau}, \ZZ)$ is injective. Indeed, considering a deformation $\mathcal{K}\to B$ as above with generic element of Picard rank~$0$, we see that the pull-back $\iota^*_{\tau}$ must be either zero or injective, since $H^2(\mathcal{K}_b,\QQ)$ is an irreducible Hodge structure for very general $b$ and the integral second cohomology of $\mathrm{Kum}^3$-manifolds is torsion-free. But this map cannot be zero, as the restriction of a symplectic form on~$K^3(A)$ gives a symplectic form on $W_{\tau}$. Similar statements hold for the K3 surfaces $\iota_{\tau, \tau'}\colon V_{\tau,\tau'}\to K^3(A)$. 
		\end{remark}
	
	\subsection{Divisors on the canonical submanifolds} 
	For $\tau\neq\tau'\in A_2$ we have the identification $V_{\tau, \tau'} = \Km^{\tau} (A/\langle \tau+\tau'\rangle)$; thus $V_{\tau, \tau'}$ contains $16$ disjoint rational curves.
	
	\begin{definition}\label{def:curvesR}
		For $\beta\in (A_{2,\tau}\sqcup A_{2,\tau'})/\langle \tau+\tau'\rangle$, we let $R_{\tau,\tau',\beta}\subset V_{\tau,\tau'}$ be the rational curve lying over the node $\beta$ of $ A/\langle (\tau,-1), (\tau', -1)\rangle$. We let~$r_{\tau,\tau',\beta}\in H^2(V_{\tau,\tau'}, \ZZ)$ be the fundamental class of $R_{\tau,\tau',\beta}$.
	\end{definition}

	Let $\tau\in A_{2}$.
	By Proposition \ref{prop:relevantSubvarieties} we have $W_{\tau} = \Km^\tau (A)^{[2]}$ and $\nu$ factors through $$\bar{\nu} \colon \Km^{\tau} (A)^{[2]}\to (A/\langle (\tau,-1)\rangle)^{(2)},$$ where the latter is immersed in $A^{(4)}_0$ via $(a,b )\mapsto (a,b,-a+\tau, -b+\tau)$. The morphism $\bar{\nu}$ introduces $17$ exceptional divisors on $W_{\tau}$.
	
	\begin{definition} \phantomsection \label{def:divisors}
		Fix $\tau\in A_2$.
		\begin{itemize}
			\item  For any $\alpha\in A_{2,\tau}$ we let $S_{\tau, \alpha}\subset W_{\tau}$ be the exceptional component of $\bar{\nu}$ lying over $
	\{ (\alpha, b), \ b\in A \} \subset (A/\langle (\tau,-1)\rangle)^{(2)}
	$; we let $s_{\tau,\alpha}\in H^2(W_{\tau},\ZZ)$ be its fundamental class.
	\item We let $D_{\tau}\subset W_{\tau}$ be the divisor of non-reduced subschemes, which is the exceptional component of $\bar{\nu}$ over~$
	\{ (a, a) , \ a \in A\} \subset (A/\langle (\tau,-1)\rangle)^{(2)}
    $. We let $\delta_\tau\in H^2(W_{\tau},\ZZ)$ be half the fundamental class of $D_{\tau}$.
	\end{itemize}
    \end{definition}

	By \cite[Proposition 6]{beauville1983varietes}, the class $\delta_{\tau}$ is integral. Moreover, denoting by $\pi^{\tau}\colon A\dashrightarrow \Km^{\tau}(A)$ the natural rational map of degree~$2$, we have (up to finite index, see \cite[Chapter 3, \S2.5]{huyK3}) 
	\begin{align*}
		H^2(W_{\tau}, \ZZ) 
		& = \pi^{\tau}_* (H^2(A,\ZZ)) \oplus \bigoplus_{\alpha\in A_{2,\tau}} \ZZ\cdot s_{\tau, \alpha} \oplus \ZZ\cdot \delta_{\tau}.
	\end{align*}
	With respect to the Beauville-Bogomolov form on $W_{\tau}$, the classes $\delta_{\tau}$ and~$s_{\tau, \alpha}$, $\alpha\in A_{2,\tau}$, are pairwise orthogonal and have square $-2$.
	  \begin{remark}\label{rmk:actionDivisors}
		Let $\epsilon\in A_4$. Translation by $\epsilon$ on $K^3(A)$ sends $D_{\tau}$ to $D_{\tau+2\epsilon}$ and $S_{\tau,\alpha}$ to~$S_{\tau+2\epsilon, \alpha+\epsilon}$, for any $\tau\in A_2$ and $\alpha\in A_{2,\tau}$.
	\end{remark}
	
	We will now make some further observations about the geometry of the submanifolds introduced above. 
	
	\begin{proposition}Let $\tau\neq \tau'\in A_2$ and consider $V_{\tau, \tau'}\subset W_{\tau}$. \label{prop:geometricInput1}
		\begin{enumerate}[label=(\roman*)]
			\item Restriction from $W_{\tau}$ to $V_{\tau, \tau'}$ gives in $H^2(V_{\tau, \tau'}, \ZZ)$:
			\begin{align*}
				{\delta_{\tau}}_{|_{V_{\tau, \tau'}}} = \frac{1}{2} \sum_{\gamma\in A_{2,\tau'}/\langle \tau+\tau'\rangle} r_{\tau, \tau', \gamma}, \ \ \ \ \text{ and } \ \ \ \ 
				{s_{\tau, \alpha}}_{|_{V_{\tau, \tau'}}} = r_{\tau, \tau', \bar{\alpha}} \ \ \text{for any } \alpha\in A_{2,\tau},
			\end{align*}
			where $\bar{\alpha}\in A_{2,\tau}/\langle \tau+\tau'\rangle$ denotes the image of $\alpha$.
			\item Let $N_{V_{\tau, \tau'}|W_{\tau}}$ be the normal bundle of $V_{\tau,\tau'}\subset W_{\tau}$. We have $$\deg\left(c_2\bigl(N_{V_{\tau, \tau'}|W_{\tau}}\bigr)\right)=12.$$
		\end{enumerate} 
	\end{proposition}
	\begin{proof} 
		It is well-known that there is a commutative diagram
		\[
		\begin{tikzcd}
			\mathrm{Bl}_\Delta (\Km^{\tau}(A)^2) \arrow[swap]{d}{\rho} \arrow{r}{\pi} & \Km^{\tau}(A)^2 \arrow{d} \\
			\Km^{\tau}(A)^{[2]} \arrow{r}{\bar{\pi}} & \Km^{\tau}(A)^{(2)}
		\end{tikzcd} 
		\]
		where $\Delta\subset \Km^{\tau}(A)^2$ is the diagonal, $\pi$ is the blow-up map, and $\rho$ is a double cover branched over the exceptional divisor of $\pi$. The lower horizontal map $\bar{\pi}$ is the blow-up of the image $\bar{\Delta}$ of the diagonal in $\mathrm{Km}^\tau(A)^{(2)}$.		
		There is a natural action of $A_2$ on $\Km^{\tau}(A)$ and hence on $\Km^{\tau}(A)^{[2]}$; the inclusion of $V_{\tau, \tau'}$ into $W_{\tau}$ is induced by the morphism 
		\[
		\psi\colon \Km^{\tau}(A) \to \Km^{\tau}(A)^2 
		\]
		given by $s \mapsto (s, (\tau +\tau')(s))$. We denote by $Z_{\tau+\tau'}$ the image, i.e., the graph of the automorphism of $\Km^{\tau}(A)$ induced by translation by $\tau+\tau'$ on $A$.
		
		We prove $(i)$. Let $\bar{Z}_{\tau+\tau'}$ be the image of $Z_{\tau+\tau'}$ in $ \Km^{\tau}(A)^{(2)}$. Then the restriction $\bar{\pi}\colon V_{\tau, \tau'} \to \bar{Z}_{\tau+\tau'}$ is the blow-up of the intersection of $\bar{Z}_{\tau+\tau'}$ with~$\bar{\Delta}$, which consists of $8$ points parametrized by $A_{2,\tau'}/ \langle \tau+\tau' \rangle$. 
		Since $D_{\tau}\subset W_{\tau}$ is the exceptional divisor of~$\bar{\pi}$, we deduce that it restricts to $\sum_{\beta\in A_{2,\tau'}/\langle \tau+\tau'\rangle} r_{\tau, \tau', \beta}$; as $[D_{\tau}]=2\delta_{\tau}$, this gives the first assertion. 
		For the second assertion of $(i)$, note that $S_{\tau,\alpha}\subset \Km^{\tau}(A)^{[2]}$ parametrizes the subschemes whose support intersects the rational curve of $\Km^{\tau}(A)$ lying over the node corresponding to the image $\bar{\alpha}\in A_{2,\tau}/\langle \tau+\tau'\rangle $ of $\alpha\in A_{2,\tau}$. 
		The intersection of $\bar{\pi}(S_{\tau,\alpha})$ with $\bar{Z}_{\tau+\tau'}$ is disjoint from $\bar{\Delta}$, and it is easy to see in~$\Km^{\tau}(A)^{(2)}$ that the restriction of $\bar{\pi} (S_{\tau,\alpha})$ to $\bar{\pi}(V_{\tau,\tau'})$ is the curve $\bar{\pi}(R_{\tau,\tau',\bar{\alpha}})$.
		
		We now prove $(ii)$. The normal bundle of the graph $Z_{\tau+\tau'}\subset \Km^{\tau}(A)^2$ is isomorphic to the tangent bundle of the K3 surface $\Km^{\tau}(A)$ and hence, by \cite[Corollary 6.3]{fulton},
		$$ \int_{\Km^{\tau}(A)^2} [Z_{\tau+\tau'}]^2 = \deg \left(c_2\bigl(N_{Z_{\tau+\tau'}|\Km^{\tau}(A)^2}\bigr)\right) =24.$$
		Note that $Z_{\tau + \tau'}$ is transverse to $\Delta$ and it is stabilized by the involution of $\Km^{\tau}(A)^2$ which switches the two factors.  By \cite[Theorem 6.7]{fulton}, we have $\rho^*([V_{\tau,\tau'}]) = \pi^* ([Z_{\tau+\tau'}])$ in the cohomology of $\mathrm{Bl}_{\Delta}(\Km^{\tau}(A)^2)$, and therefore
		\begin{align*}
			\int_{W_{\tau}} [V_{\tau,\tau'}]^2 & = \frac{1}{2} \int_{\mathrm{Bl}_\Delta (\Km^{\tau}(A)^2)} \rho^*([V_{\tau, \tau'}]^2) 
			 = \frac{1}{2}\int_{\mathrm{Bl}_\Delta (\Km^{\tau}(A)^2)} \pi^*([Z_{\tau+\tau'}])^2
			=12;
		\end{align*}
		this self-intersection number coincides with the degree of $c_2(N_{V_{\tau, \tau'}|W_{\tau}})$.
	\end{proof}

	By \cite[Proposition 8]{beauville1983varietes}, we have 
	$$H^2(K^3(A),\ZZ)=H^2(A,\ZZ)\oplus \ZZ\cdot \xi, $$
	where $\xi$ has self-intersection $-8$ and $2\xi = [E]$ is the fundamental class of the exceptional divisor $E$ of the Hilbert-Chow morphism $\nu\colon K^3(A)\to A^{(4)}_0$. 
		
	\begin{proposition} \label{prop:geometricInput}
		 Let $\tau\neq \tau'\in A_2$ and consider $V_{\tau, \tau'}\subset W_{\tau}\subset K^3(A)$. We have:
			\begin{align*}
				H^2(W_{\tau}, \ZZ) \ni \phantom{\tau} \xi_{|_{W_{\tau}}} & =  2\delta_{\tau} + \frac{1}{2} \sum_{\alpha\in A_{2,\tau}} s_{\tau, \alpha},\\
				H^2(V_{\tau, \tau'},\ZZ) \ni \xi_{|_{V_{\tau,\tau'}}} & =  \sum_{\beta\in A_{2,\tau}/\langle \tau+\tau'\rangle} r_{\tau, \tau', \beta} + \sum_{\gamma\in A_{2,\tau'}/\langle \tau+\tau'\rangle} r_{\tau, \tau', \gamma}.
			\end{align*} 
	\end{proposition}
	\begin{proof}
	 The intersection $E\cap W_{\tau} $ is supported on the union of the divisors $S_{\tau, \alpha}$ and~$D_{\tau}$.
	 First, we observe that the restriction of $E$ to $W_{\tau}$ has multiplicity $1$ along each $S_{\tau,\alpha}$. 
	 Indeed, each $S_{\tau, \alpha}\subset W_{\tau}$ is generically a $\PP^1$-bundle over $\{ (\alpha, b, \alpha, -b+\tau), \ b\in A\} \subset A_0^{(4)}$, an open subset of which is contained in the locus $U$ of those $(a_1, a_2, a_3, a_4)\in A_0^{(4)}$ with at least three distinct points in their support. 
	 The restriction $\nu\colon \nu^{-1}(U)\to U$ is the blow-up of the big diagonal. Since $\nu(D_{\tau})\cap U=\emptyset$, the restriction of $\nu$ to $\nu^{-1}(U)\cap W_{\tau}$ is a blow-up with exceptional divisors the~$S_{\tau,\alpha}$. Hence, $E_{|_{W_{\tau}}} =kD_{\tau} + \sum_{\alpha\in A_{2,\tau}} S_{\tau,\alpha}$, for some~$k$.
	 
	 We now compute the restriction of $E$ to $V_{\tau, \tau'}$. It is supported over the $16$ curves $R_{\tau, \tau', \beta}$, for $\beta\in (A_{2, \tau} \sqcup A_{2, \tau'})/\langle \tau+\tau'\rangle$.
	 By the above and Proposition \ref{prop:geometricInput1}.$(i)$, the restriction of~$E$ to~${V_{\tau, \tau'}}$ has multiplicity $2$ at the curve $R_{\tau, \tau', \bar{\alpha}}$ for any $\bar{\alpha}\in A_{2,\tau}/ \langle \tau+\tau' \rangle$ (recall that~$S_{\tau, \alpha}$ and~$S_{\tau, \alpha+\tau+\tau'}$ restrict to the same curve $R_{\tau, \tau', \bar{\alpha}}$).
	 Switching the roles of $\tau$ and $\tau'$, we find $$\xi_{|_{V_{\tau, \tau'}}} = \sum_{\beta\in A_{2,\tau}/\langle \tau+\tau'\rangle} r_{\tau, \tau', \beta} + \sum_{\gamma\in A_{2,\tau'}/\langle \tau+\tau'\rangle} r_{\tau, \tau', \gamma} .$$
	 
	 Finally, combining this with Proposition \ref{prop:geometricInput1}.$(i)$, the multiplicity $k$ of $E\cap W_{\tau}$ along $D_{\tau}$ must be $2$; otherwise, $\xi$ could not restrict to the above class in $H^2(V_{\tau, \tau'}, \ZZ)$. 
	\end{proof}
	
	\begin{remark}\label{rmk:mixedIntersections}
		Let $\iota_{\tau}\colon W_{\tau}\hookrightarrow K^3(A)$ denote the inclusion.
		If $\nu\colon K^3(A)\to A^{(4)}_0$ is the Hilbert-Chow morphism, an explicit computation shows that
		$\nu(\iota_{\tau}(S_{\tau, \alpha}))\cap \nu(\iota_{\tau'}(S_{\tau', \alpha'}))=\emptyset$ and $\nu(\iota_{\tau}(D_{\tau}))\cap \nu(\iota_{\tau'}(D_{\tau'}))=\emptyset $, for any $\tau\neq\tau'\in A_{2}$. Thus, for any $\tau\neq\tau'$, we have
		\[
		\int_{K^3(A)} 
		\iota_{\tau, *}(s_{\tau,\alpha}) \cdot \iota_{\tau',*}(s_{\tau', \alpha'})=0 \ \ \ \ \text{and} \ \ \ \ 
		\int_{K^3(A)} 
		\iota_{\tau, *}(\delta_{\tau}) \cdot \iota_{\tau',*}(\delta_{\tau'})=0.
		\]
	\end{remark}
	\begin{remark}\label{rmk:normalBundles}
		The bundles in the exact sequence 
		\[
		0\to T_{W_{\tau}} \to {T_{K^3(A)}}_{|_{W_{\tau}}} \to N_{W_{\tau}| K^3(A)}\to 0
		\]
		have trivial first Chern class. Hence, $c_2\left(K^3(A)\right)_{|_{W_{\tau}}} = c_2 \left(W_{\tau}\right) + c_2 \left(N_{W_{\tau}| K^3(A)}\right)$, and 
		\begin{equation}\label{eq:c2NW}
			c_4\left(K^3(A)\right)_{|_{W_{\tau}}}  = c_4\left(W_{\tau}\right) + c_2\left(W_{\tau}\right) \cdot c_2\left(N_{W_{\tau}| K^3(A)}\right). 
		\end{equation}
		We also have an exact sequence
		$$0 \to N_{V_{\tau, \tau'}| W_{\tau}} \to {N_{V_{\tau, \tau'}|K^3(A)}} \to {N_{W_{\tau}|K^3(A)}}_{|_{V_{\tau,\tau'}}} \to 0  $$
		of bundles with trivial first Chern class.	
		Since $V_{\tau,\tau'}$ is the transverse intersection of~$W_{\tau}$ and~$W_{\tau'}$, the last term is identified with $N_{V_{\tau,\tau'}|W_{\tau'}}$. Moreover, translation by $\epsilon\in A_{2,\tau+\tau'}$ maps $V_{\tau,\tau'}$ to itself and exchanges~$W_{\tau}$ with~$W_{\tau'}$, inducing an isomorphism of vector bundles~$N_{V_{\tau,\tau'}|W_{\tau}}\cong N_{V_{\tau,\tau'}|W_{\tau'}}$. Via the normal bundle sequence for the embedding of~$V_{\tau, \tau' } $ into~$ K^3(A)$, we find
		\begin{align*}
			{c_2(K^3(A))}_{|_{V_{\tau, \tau'}}} & 
			 = c_2\bigl(V_{\tau, \tau'}\bigr)+ c_2\bigl( N_{V_{\tau, \tau'}| W_{\tau}}\bigr)+ c_2\left(N_{V_{\tau, \tau'}| W_{\tau'}}\right).
		\end{align*}
		Since $V_{\tau, \tau'}$ is a K3 surface, Proposition \ref{prop:geometricInput1}.(ii) implies that 
		\begin{equation} \label{eq:degreeC2}
			 \deg\left({c_2(K^3(A))}_{|_{V_{\tau, \tau'}}}\right) = 48.
		\end{equation} 
	\end{remark}
	
	\section{Canonical Hodge classes of degree four} \label{sec:4}
	 
	 	In this section we will show that canonical Hodge classes in $H^4(K,\QQ)$ are analytic, for any manifold $K$ of $\mathrm{Kum}^3$-type. As in the previous section, let $A$ be a complex torus of dimension $2$ and let $K^3(A)$ be the generalized Kummer sixfold on $A$. Recall its submanifolds~$W_{\tau}$ and $V_{\tau, \tau'}$ from Definition \ref{def:relevantSubvarieties}.
	 	\begin{definition}
			We denote by $w_{\tau}\in H^4(K^3(A),\ZZ)$ (resp.\ by $v_{\tau, \tau'}\in H^8(K^3(A), \ZZ)$) the fundamental class of $W_{\tau}\subset K^3(A)$ (resp.\ of $V_{\tau,\tau'}$), for $\tau\neq\tau'$ in $A_2$. 
	 	\end{definition}
 		Note that, for any $\tau\neq\tau'\in A_2$, we have $w_{\tau}\cdot w_{\tau'} = v_{\tau, \tau'}$, because $W_{\tau}$ and $W_{\tau'}$ meet transversely in $V_{\tau, \tau'}$ by Proposition \ref{prop:relevantSubvarieties}.$(iii)$. 
	 
	 	\begin{lemma}\label{lem:deformingW}
	 		The classes $w_{\tau}$ and $v_{\tau, \tau'}$ are canonical Hodge classes, and moreover they remain analytic on any deformation of $K^3(A)$.
	 	\end{lemma}
 	\begin{proof}
 	 	The Lemma follows immediately from Remark \ref{rmk:deformations}.
 	 	\end{proof}
	
	The Proposition below will then imply that canonical Hodge classes of degree $4$ are analytic on any deformation of $K^3(A)$. 
	\begin{proposition} \label{prop:canonicalH4}
		Any canonical Hodge class in $H^4(K^3(A),\QQ)$ is a linear combination of the $17$ linearly independent classes $c_2$ and $w_{\tau}$, for $\tau\in A_2$.
	\end{proposition} 	
	In fact, since we know from \S\ref{subsec:HodgeCanonical} that the space of canonical Hodge classes has rank $17$, it will be sufficient to check that the classes in the statement are linearly independent.
	In turn, this will be proven using the following intersection numbers.
	
	\begin{proposition}\label{prop:intersectionsW}
		We have 
		\[
		w_{\tau}^3= 60, \ \ \ \ \  w_{\tau}^2\cdot w_{\tau'}=12, \ \ \ \ \   w_{\tau}\cdot w_{\tau'}\cdot w_{\tau''} = 4,
		\]
		for any $\tau, \tau',\tau''$ in $A_2$ pairwise distinct.
	\end{proposition}
	
	Part of the statement follows directly from Propositions \ref{prop:geometricInput1} and \ref{prop:geometricInput}.
		
	\begin{lemma}\label{lem:firstIntersections}
		The generalized Fujiki constants are 
		\[
		C(w_{\tau}) = 12 \ \ \ \text{ and } \ \ \ C(w_{\tau}\cdot w_{\tau'}) = 4,
		\]
		for any $\tau\neq \tau' \in A_2$. For pairwise distinct elements $\tau, \tau', \tau''$ of $A_{2}$ we have
		\[
		 c_2\cdot w_{\tau} \cdot w_{\tau'} = 48;\  \ \ \ \   w^2_\tau\cdot w_{\tau'} = 12; \ \ \ \ \  w_{\tau}\cdot w_{\tau'}\cdot w_{\tau''} = 4.
		\]
	\end{lemma}
	\begin{proof}
		Let $\xi\in H^2(K^3(A),\ZZ)$ be half the class of the Hilbert-Chow divisor $E$; recall that $q_{K^3(A)}(\xi,\xi)=-8$. The Fujiki constant of $\mathrm{K}3^{[2]}$-manifolds equals $3$ (see \cite{rapagnetta2008beauville}). Therefore, Fujiki relation \eqref{eq:fujiki}, the projection formula, and Proposition \ref{prop:geometricInput} give 
		\begin{align*}
			C(w_{\tau})\cdot 64 & = \int_{K^3(A)} w_\tau \cdot \xi^4  = \int_{W_{\tau}} \left(\xi_{|_{W_{\tau}}}\right)^4  = 3\cdot q_{W_{\tau}} \left(\xi_{|_{W_{\tau}}}, \xi_{|_{W_{\tau}}}\right)^2  = 3 \cdot 256, 
		\end{align*}
		from which we find $C(w_{\tau}) =12$. Similarly, by Proposition \ref{prop:geometricInput1}.$(i)$ we have
		\begin{align*}
			C(w_{\tau}\cdot w_{\tau'} ) \cdot (-8)  & = \int_{K^3(A)} v_{\tau, \tau'} \cdot \xi^2  = \int_{V_{\tau,\tau' }} \left(\xi_{|_{V_{\tau, \tau'}}}\right)^2 = -32 ,
		\end{align*}
		which leads to $C(w_\tau\cdot w_{\tau'} )=4$.
		By Proposition \ref{prop:relevantSubvarieties}, three distinct $W_{\tau}, W_{\tau'}, W_{\tau''}$ meet transversely in $4$ points, and hence $w_{\tau}\cdot w_{\tau'} \cdot w_{\tau''}=4$.
		The self-intersection of $V_{\tau, \tau'}$ as a submanifold of $W_{\tau}$ equals the degree of $c_2\bigl(N_{V_{\tau,\tau'}|W_{\tau}}\bigr)$, which is~$12$ by Proposition \ref{prop:geometricInput1}.$(ii)$. The restriction of $w_{\tau'}$ to $W_{\tau}$ is the fundamental class of $V_{\tau,\tau'}$ in $H^4(W_{\tau},\ZZ)$, and therefore  
		$$ w^2_{\tau'} \cdot w_{\tau} =\int_{W_{\tau}} [V_{\tau, \tau'}]^2  = 12. $$
		Finally, by \eqref{eq:degreeC2}, we have 
		$c_2\cdot v_{\tau, \tau'} = 48$.
	\end{proof}
		
	\subsection{The $\Gamma$-invariant classes}
	In order to prove Proposition \ref{prop:intersectionsW}, we consider
	$$w\coloneqq \sum_{\tau\in A_2} w_{\tau} \in H^4(K^3(A), \ZZ) \  \ \ \ \text{and} \ \ \ \ v=\frac{1}{2} \sum_{\tau\neq\tau'\in A_2} v_{\tau, \tau'} \in H^8(K^3(A),\ZZ). $$
	By Proposition~\ref{prop:groupAction}, these canonical Hodge classes are invariant under the action of $\Gamma=A_4$. We shall express them in terms of the basis of Lemma \ref{lem:basis}. 
	
	\begin{lemma}\label{lem:classofW}
		We have 
		$$w=\tfrac{16}{11}\overline{q} - 3z \ \ \ \ \ \text{ and } \ \ \ \ \ v= \tfrac{40}{33}\overline{q}^2 - \tfrac{45}{7} \overline{q} \cdot z. $$
	\end{lemma}
	\begin{proof}
		By the above Lemma \ref{lem:firstIntersections}, we have 
		$$C(v)=120\cdot 4, \ \ \ \ \ C(w)= 16 \cdot 12, \ \ \ \ \  c_2 \cdot v  = 120 \cdot 48.$$ 
		Write 
		$$v= \tfrac{C(v)}{C(\overline{q}^2)} \overline{q}^2 + \gamma \overline{q}\cdot z = \tfrac{40}{33} \overline{q}^2 + \gamma \overline{q}\cdot z, $$
		for some coefficient $\gamma\in\QQ$.
		Then, via Lemma \ref{lem:basis}, we find
		\begin{align*}
		120\cdot 48 = c_2 \cdot v & = \left(\tfrac{24}{11}\overline{q} + z\right)\cdot \left(\tfrac{40}{33} \overline{q}^2 + \gamma \overline{q}\cdot z\right) \\
		& =\tfrac{24}{11} \cdot \tfrac{40}{33} \overline{q}^3 + \gamma \overline{q}\cdot z^2 \\
		& = \tfrac{80640}{11} + \gamma \tfrac{2688}{11},	
		\end{align*}
		which gives $\gamma= - \tfrac{45}{7}$.
		We now use Lemma \ref{lem:firstIntersections} to compute the intersection
		\begin{align*}
		 w\cdot v & = \left(\sum_{\tau} w_{\tau}\right)\cdot \left(\frac{1}{2}\sum_{\tau\neq \tau'} {w_{\tau} \cdot w_{\tau'}}\right) \\
		& = \sum_{\tau\neq \tau'} w_{\tau}^2 \cdot w_{\tau'} + \frac{1}{2}\sum_{\tau\neq \tau'\neq \tau''\neq \tau} w_{\tau}\cdot w_{\tau'} \cdot w_{\tau''}
		\\
		& = 16 \cdot 15 \cdot 12 + \frac{1}{2} \cdot 16 \cdot 15 \cdot 14 \cdot 4\\
		& = 9600. 
		\end{align*}
	We can write $$w=\tfrac{C(w)}{C(\overline{q})} \overline{q} + \lambda z = \tfrac{16}{11}\overline{q} + \lambda z $$
	for some coefficient $\lambda\in\QQ$, which is then determined by the equation 
	\[
	9600=  w\cdot v = \left(\tfrac{16}{11}\overline{q} + \lambda z\right)\cdot \left(\tfrac{40}{33}\overline{q}^2 - \tfrac{45}{7} \overline{q}\cdot z\right) = \tfrac{53760}{11} - \tfrac{17280}{11} \lambda;
	\]
	we find $\lambda= -3$.		
	\end{proof}
	\begin{remark}
		We obtain the integral linear relations
		$$w = 8\overline{q} -3c_2 \ \ \text{ and } \ \ 3v=7c_4-c_2^2. $$
	\end{remark}
	
	We can now complete the proof of the main results of this section. 
	\begin{proof}[Proof of Proposition \ref{prop:intersectionsW}]
	By Lemma \ref{lem:firstIntersections}, it remains to prove that $w_{\tau}^3=60$.
	By Proposition \ref{prop:groupAction}, the action of $\Gamma$ is transitive on the classes $w_{\tau}$ and hence the number $w_{\tau}^3$ does not depend on $\tau\in A_2$. 
	Via the expression for $w$ found in Lemma \ref{lem:classofW} above and the constants determined in Lemma \ref{lem:basis} we calculate $w^3=23040$. 
	Expanding $w^3 = (\sum_{\tau} w_{\tau})^3$ and using the intersection numbers already computed in Lemma \ref{lem:firstIntersections}, we find the equation 
	\begin{align*}
		23040 & = \sum_{\tau} w_{\tau}^3 + 3 \sum_{\tau\neq \tau'}w_{\tau}^2 \cdot w_{\tau'} + \sum_{\tau\neq\tau'\neq\tau''\neq\tau} w_{\tau}\cdot w_{\tau'}\cdot w_{\tau''}\\
		&=16\cdot w_{\tau}^3 + 3 \cdot 16 \cdot 15 \cdot 12 + 16\cdot 15\cdot 14\cdot 4\\
		&= 16\cdot w_{\tau}^3 + 22080,
	\end{align*}
	which gives $w_{\tau}^3=60$.
	\end{proof}
	\begin{proof}[Proof of Proposition \ref{prop:canonicalH4}]
		Assume given a linear combination $u=\sum_{\tau\in A_2} \beta_{\tau} w_{\tau}$ in the classes~$w_{\tau}$ such that $u=0$. By Proposition \ref{prop:intersectionsW}, we have 
		\[
		u \cdot (w^2_{\tau} - w^2_{\tau'}) =  (60 - 12) (\beta_{\tau} - \beta_{\tau'})=0,
		\]
		and therefore $\beta_{\tau}=\beta_{\tau'}$ for any $\tau, \tau'\in A_2$. Thus $u=\beta w$, and hence $\beta=0$, which means that the $w_{\tau}$ are linearly independent. Since the $\Gamma$-invariant classes in the subspace generated by the $w_{\tau}$ are the multiples of $w$, if $c_2$ were in this subspace we would obtain a non-trivial linear relation between the independent classes $\overline{q}$ and $c_2$, a contradiction.
	\end{proof} 
	\begin{remark}\label{rmk:someValues}
		For later use we compute that, for any $\tau\in A_2$, we have
		$$C(w_\tau^2)= 12 \ \ \text{and} \ \ c_4\cdot w_{\tau} = 408 .$$
		Indeed, from the expression for $w$ found in Lemma \ref{lem:classofW} and thanks to Lemma \ref{lem:basis} we find~$C(w^2)=1152$. Expanding the square of $w$ we obtain
		$$1152 = 16\cdot C(w_{\tau}^2) + 16\cdot 15\cdot C(v_{\tau, \tau'})=16\cdot C(w_{\tau}^2) + 16\cdot 15\cdot 4,$$
		which yields $C(w_{\tau}^2)=12$, as $C(v_{\tau,\tau'})=4 $ by Lemma \ref{lem:firstIntersections}.
		We also easily calculate via Lemma \ref{lem:basis} 
		$$c_4\cdot w_{\tau} = \tfrac{1}{16} c_4\cdot w = \tfrac{1}{16} \cdot \left(\tfrac{40}{33}\overline{q}^2 - \tfrac{47}{21} \overline{q}\cdot z\right) \cdot \left(\tfrac{16}{11}\overline{q}-3z\right)=408.$$
	\end{remark}
	\begin{remark}\label{rmk:productq}
		Cup-product with $\overline{q}_{K^3(A)}$ is injective on the space of canonical Hodge classes in $H^4(K^3(A),\QQ)$. This follows via an argument entirely analogous to that used in the proof of Proposition \ref{prop:canonicalH4}, using the intersection numbers 
		\[
		\overline{q} \cdot w_{\tau}^2 = 84 \ \ \text{ and } \ \ \overline{q}\cdot w_{\tau}\cdot w_{\tau'} = 28,
		\] which are calculated from Lemma \ref{lem:firstIntersections} via \cite[Proposition 2.4]{BS22}.
	\end{remark}

	\section{Restriction of canonical Hodge classes}
	The main result of this section is Proposition \ref{prop:c2normal}, which gives the second Chern class of the normal bundle of $\iota_{\tau}\colon W_{\tau}\hookrightarrow K^3(A)$ as an element of the cohomology of $W_{\tau}$. 
	In the next section we will use this to show that canonical Hodge classes of degree $6$ are analytic. 
	
	\subsection{Cohomology of $\mathrm{K}3^{[2]}$-manifolds}\label{subsec:K3[2]}
	Following \cite{OG08}, for any manifold $X$ of $\mathrm{K}3^{[2]}$-type we have 
	\[H^4(X , \QQ) = \mathrm{Sym}^2 (H^2(X, \QQ)). \]
	The intersection product defines a non-degenerate symmetric bilinear form on the space~$\mathrm{Sym}^2(H^2(X, \QQ))$, determined by 
	\begin{equation}\label{eq:intersectionForm}
		(a_1 a_2, a_3 a_4) = q(a_1, a_2)\cdot q(a_3, a_4) + q(a_1, a_3)\cdot q(a_2, a_4) + q(a_1, a_4)\cdot q(a_2, a_3)
	\end{equation}
	where $q=q_{X}$ is the Beauville-Bogomolov form. The Euler characteristic of $X$ is $\deg(c_4(X))=324$.
	The class~$\overline{q}_{X}\in H^4(X, \QQ)$ given by the Beauville-Bogomolov form satisfies $c_2(X)=\frac{6}{5} \cdot \overline{q}_{X}$ and $\int_X \overline{q}^2_X = 575$.
 	The generalized Fujiki constant is $C(\overline{q}_X)=25$; this means that, for every $\alpha,\beta\in H^2(X,\QQ)$, we have
	\begin{equation}\label{eq:FujikiW}
		\int_{X} \overline{q}_X \cdot \alpha\cdot \beta = 25\cdot q_{X}(\alpha,\beta).
	\end{equation}	
	
	\subsection{The cohomology of $W_{\tau}$}
	Let $A$ be a $2$-dimensional complex torus and let~$K^3(A)$ be the generalized Kummer variety on it. 
	For any $\tau\in A_2$, we consider the submanifold~$\iota_{\tau}\colon W_{\tau}\hookrightarrow K^3(A)$ (Definition \ref{def:relevantSubvarieties}); by Proposition \ref{prop:relevantSubvarieties}, $W_{\tau}=\Km^{\tau}(A)^{[2]}$. 
	Recall from Definition \ref{def:divisors} that we defined divisor classes $s_{\tau,\alpha}$ and $\delta_{\tau}$ on $W_{\tau}$, which are pairwise orthogonal $-2$ classes in the lattice $H^2(W_{\tau},\ZZ)$.
	
\begin{lemma}\label{lem:intersectionMatrix}
	Consider the $19$-dimensional subspace of $H^4(W_{\tau},\QQ)$ generated by:
	\[
	\overline{q}_{W_{\tau}}; \ \ \ \delta_{\tau}^2; \ \ \sum_{\alpha\in A_{2, \tau}} s_{\tau, \alpha}^2;
	\ \ \ \sum_{\alpha\in A_{2, \tau}} s_{\tau, \alpha} \cdot s_{\tau, \alpha +\theta}\  \text{ for } 0\neq \theta\in A_2; \ \ \ \delta_{\tau} \cdot \sum_{\alpha\in A_{2,\tau}} s_{\tau, \alpha}. 
	\] 
	The intersection form on this subspace is given (in the above ordered basis) by the matrix: 
	\[
	\begin{pmatrix}
		575 & -50 & -800 & 0 &  \cdots  &  0 & 0 \\
		-50 & 12 & 64 & 0  &  \cdots  & 0 &  0 \\
		-800 & 64 & 1152 & 0 & \cdots  & 0 & 0 \\
		0 & 0 & 0 & 128 & \ddots & 0  &  0 \\
		\vdots & \vdots & \vdots & \ddots & \ddots & \ddots &  \vdots  \\
		0 & 0 & 0  & 0 & \ddots & 128  & 0  \\ 
		0 & 0 & 0 & 0 & \cdots & 0 & 64  
	\end{pmatrix}
	\]
\end{lemma}
\begin{proof}
	The calculation is elementary using the description \eqref{eq:intersectionForm} of the intersection product and the Fujiki relation \eqref{eq:FujikiW}. 
	One computes
	\[
	\overline{q}_{W_{\tau}}\cdot \delta_{\tau}^2=-50; \ \ \ \overline{q}_{W_{\tau}}\cdot \delta_{\tau}\cdot s_{\tau,\alpha}=0; \ \ \
	\overline{q}_{W_\tau}\cdot s_{\tau,\alpha}\cdot s_{\tau,\alpha'} =
	\begin{cases} -50 & \text{if} \ \alpha=\alpha';\\
		0 & \text{else}.
	\end{cases}
	\]
	Next, we have	 
	\begin{align*}
		(s_{\tau, \alpha}\cdot s_{\tau, \beta}) \cdot (s_{\tau, \alpha'}\cdot s_{\tau, \beta'}) & = \begin{cases}
			12 & \text{ if } \alpha=\alpha'=\beta=\beta';\\
			4 & \text{ if } \alpha\neq \beta \text{ and } \{\alpha, \alpha'\}= \{\beta=\beta'\}; \\
			4 & \text{ if } \alpha= \beta, \alpha'=\beta', \alpha\neq \alpha';\\
			0 & \text{ else};
		\end{cases}	
	\end{align*}
and
	\[
	\delta_{\tau}^4 = 12;\ \ \ \delta_{\tau}^3\cdot s_{\tau,\alpha}=0; \ \ \ \ \delta_{\tau}^2 \cdot (s_{\tau,\alpha} \cdot s_{\tau,\beta})=\begin{cases}4 &\text{ if } \alpha=\beta; \\0 &\text{ else.}
	\end{cases}
	\] 
By the bilinearity of the intersection product, the matrix in the statement is now obtained via a straightforward calculation.
\end{proof}

By Remark \ref{rmk:deformations}, any deformation $\mathcal{K}\to B$ induces a deformation of the $\mathrm{K}3^{[2]}$-manifold~$W_{\tau}$. However, the classes $s_{\tau,\alpha}$ and $\delta_{\tau}$ do not remain Hodge classes on all such induced deformations. We are thus led to introduce the following divisor classes.
\begin{definition}\label{def:s'}
	For any $\tau\in A_2$ and $\alpha\in A_{2,\tau}$ we define the Hodge class
	\[
	s'_{\tau, \alpha} \coloneqq (4s_{\tau, \alpha} - \delta_{\tau}) \in H^2(W_{\tau}, \ZZ).
	\]
\end{definition}
The point of the definition is explained by the next lemma. 
\begin{lemma}\label{lem:Hodge}
	Let $\mathcal{W}_{\tau}\to B$ be a deformation of $W_{\tau}$ induced by a deformation $\mathcal{K}\to B$ of $K^3(A)$. Then the classes $s'_{\tau,\alpha}$ remain Hodge on $\mathcal{W}_{\tau, b}$, for any $b\in B$.
\end{lemma}
\begin{proof}
	Each $s_{\tau,\alpha}'$ is orthogonal to $\iota^*_{\tau}(H^2(K^3(A),\ZZ))$. To prove this, we may assume that $A$ is a general complex torus such that $H^2(A,\ZZ)$ is an irreducible Hodge structure. Then $H^2(K^3(A),\ZZ)=H^2(A,\ZZ)\oplus \ZZ\cdot \xi$, and the first summand equals the transcendental part $H^2_{\mathrm{tr}}(K^3(A),\ZZ)$. Hence, $s'_{\tau,\alpha}\in\NS(W_{\tau})$ is orthogonal to~$\iota^*_{\tau}(H^2(A,\QQ))=H^2_{\mathrm{tr}}(W_{\tau},\QQ)$. By Proposition \ref{prop:geometricInput}, we have $q_{W_{\tau}}(\iota_{\tau}^*(\xi), s'_{\tau,\alpha})=0$, and we conclude that $s'_{\tau,\alpha}$ is orthogonal to $\iota_{\tau}^*(H^2(K^3(A),\ZZ))$. 
	
	Consider now a deformation as in the statement and let $s'_b \in H^2(\mathcal{W}_{\tau, b},\ZZ)$ be a class obtained via parallel transport of some $s'_{\tau, \alpha}$, for some $b\in B$. 
	By the above, $s'_{b}$ is orthogonal to the image of $\iota^*_{\tau,b}\colon H^2(\mathcal{K}_b, \ZZ) \to H^2(\mathcal{W}_{\tau,b},\ZZ)$ of the pull-back; but this image contains a symplectic form, and hence $s'_b$ is a Hodge class.
\end{proof}

	\subsection{Restriction of canonical Hodge classes}
	We will now explicitly describe the restriction to $\iota_{\tau}\colon W_{\tau}\hookrightarrow K^3(A)$ of the canonical Hodge classes in $H^4(K^3(A),\QQ)$ from Proposition \ref{prop:canonicalH4}. First, we have the following. 
	
	\begin{lemma}\label{lem:restrictCanonical}
		Let $\omega\in H^4(K^3(A), \QQ)$ be a canonical Hodge class. Then its pull-back~$\iota_{\tau}^*(\omega)\in H^4(W_{\tau}, \QQ)$ is a linear combination of the following $17$ classes: 
		\[
		\iota^*_{\tau} \left(\overline{q}_{K^3(A)}\right); \ \ \ \ \ \sum_{\alpha\in A_{2,\tau}} s'_{\tau, \alpha} \cdot s'_{\tau, \alpha+\theta}, \ \text{for } \theta\in A_2.
		\]
	\end{lemma}
	\begin{proof}
		By Proposition \ref{prop:canonicalH4}, every canonical Hodge class in $H^4(K^3(A),\QQ)$ is invariant under the action of $G=A_2\times\langle -1\rangle$. The group $G$ induces an action of $A_2$ on $W_{\tau}$, such that $\theta \in A_2$ maps $s'_{\tau,\alpha}$ to $s'_{\tau, \alpha+\theta}$ (see Remark \ref{rmk:actionDivisors}).
		
		Recall from Remark \ref{rmk:deformations} that any deformation of $K^3(A)$ induces a deformation of $W_{\tau}$. Considering the orthogonal decomposition
		\[
		H^2(W_{\tau},\QQ) = \iota^*_{\tau}(H^2(K^3(A),\QQ)) \oplus \langle s'_{\tau, \alpha} \rangle_{\alpha\in A_{2,\tau}}
		\] 
		and the induced splitting of $H^4(W_{\tau},\QQ)$, we see that the classes in $H^4(W_{\tau},\QQ)$ staying Hodge on every such deformation are the multiples of the pull-back of $\overline{q}_{K^3(A)}$ and those contained in the subspace~$\mathrm{Sym}^2 (\langle s'_{\tau, \alpha} \rangle_{\alpha\in A_{2,\tau}})$. The $A_2$-invariants in this subspace are spanned by the $16$ vectors~$\sum_{\alpha\in A_{2,\tau}} s'_{\tau, \alpha}\cdot s'_{\tau, \alpha+\theta}$, for $\theta$ varying in $A_2$. 
	\end{proof}
	\begin{remark}\label{rmk:changeBasis}
		For any $\theta\in A_2$, we have
		\begin{equation*}\label{eq:changeBasis}
			\sum_{\alpha\in A_{2,\tau}} s'_{\tau, \alpha}\cdot s'_{\tau, \alpha+\theta} = 16\cdot \delta_{\tau}^2 + 16\cdot \sum_{\alpha\in A_{2,\tau}} s_{\tau,\alpha}\cdot s_{\tau, \alpha+\theta} - 8\cdot \delta_{\tau}\cdot \sum_{\alpha\in A_{2,\tau}} s_{\tau, \alpha}.
		\end{equation*}
		In particular, the pull-back $\iota^*_{\tau}(\omega)$ of a canonical Hodge class $\omega\in H^4(K^3(A),\QQ)$ is a linear combination of the vectors of Lemma \ref{lem:intersectionMatrix}.
	\end{remark}
	
	\begin{lemma}\label{lem:restrictionBB}
		For any $\tau\in A_2$, we have in $H^4(W_{\tau},\QQ)$:
		\begin{align*}			\iota_{\tau}^*\left(\overline{q}_{K^3(A)}\right) & = 2\overline{q}_{W_{\tau}} + \frac{1}{2}\delta_{\tau}^2  +
			\frac{31}{32} \sum_{\alpha\in A_{2,\tau}} s_{{\tau}, \alpha}^2 - \frac{1}{32}\sum_{0\neq \theta\in A_2} \sum_{\alpha\in A_{2,\tau}} s_{{\tau}, \alpha}\cdot s_{{\tau},\alpha+\theta} -\frac{1}{4}\delta_{\tau}\cdot \sum_{\alpha} s_{{\tau}, \alpha} .
		\end{align*}
	\end{lemma}
	\begin{proof}
		We note that, for any $\gamma\in H^2(K^3(A), \QQ)$, we have $$q_{W_{\tau}}(\iota^*_{\tau}(\gamma), \iota^*_{\tau}(\gamma)) = 2q_{K^3(A)}(\gamma, \gamma).$$
			Indeed, as the Fujiki constant of $W_{\tau}$ is $3$, we find
			\begin{align*}
				q_{W_{\tau}} (\iota^*_{\tau} (\gamma), \iota^*_{\tau} (\gamma))^2 &= {\frac{1}{3} \int_{W_{\tau}} \iota_{\tau}^*(\gamma)^4} = {\frac{1}{3} \int_{K^3(A)} w_{\tau} \cdot \iota_{\tau}^*(\gamma)^4} = {\tfrac{C(w_{\tau})}{3} q_{K^3(A)} (\gamma, \gamma)^2},
			\end{align*}
			which equals $4 q_{K^3(A)} (\gamma, \gamma)^2$ because $C(w_{\tau})=12$, by Lemma \ref{lem:firstIntersections}. As K\"ahler classes have positive Beauville-Bogomolov square, we find $q_{W_{\tau}}(\iota_{\tau}^*(\gamma), \iota_{\tau}^*(\gamma)) = 2q_{K^3(A)} (\gamma, \gamma)$.
		
		We let $\tilde{\lambda}^+_i, \tilde{\lambda}^-_i$, $i=1,2,3$, be an orthogonal basis of $H^2(A,\QQ)$ such that $\tilde{\lambda}_i^\pm$ has self-intersection $\pm 2$. The $\tilde{\lambda}_i^{\pm}$ and $\xi$ give an orthogonal basis of $H^2(K^3(A),\QQ)$, and 
		$$\overline{q}_{K^3(A)} = \frac{1}{2} \left(\sum_{i=1}^3 (\tilde{\lambda}^+_i)^2 - \sum_{i=1}^3 (\tilde{\lambda}^-_i)^2\right) -\frac{1}{8} \xi^2. $$ 
		
		Set $\lambda_i^\pm \coloneqq \iota^*_{\tau}(\tilde{\lambda}_i^{\pm})$. Then the classes $\lambda_i^+,\lambda_i^-$, $i=1,2,3$; $s_{\tau,\alpha}$, $\alpha\in A_{2,\tau}$; $\delta_{\tau}$, are an orthogonal basis of $ H^2(W_{\tau},\QQ)$, with $q_{W_{\tau}}(\lambda_i^{\pm}, \lambda_i^{\pm})=\pm 4$. Indeed, if $\pi^{\tau}\colon A\dashrightarrow \Km^{\tau}(A)$ denotes the natural rational map, we have the orthogonal decomposition
		\[
		H^2(W_{\tau},\QQ) = \pi_{*}^{\tau}(H^2(A,\QQ))\oplus\bigoplus_{\alpha\in A_{2,\tau}} \QQ\cdot s_{\tau,\alpha} \oplus \QQ\cdot \delta_{\tau}.
		\]
		The restriction $\iota^*_{\tau}\colon H^2(K^3(A),\QQ)\to H^2(W_{\tau},\QQ)$ induces an isomorphism between the first summand in $H^2(K^3(A),\QQ)=H^2(A,\QQ)\oplus\QQ\cdot\xi$ and $\pi_{*}^{\tau}(H^2(A,\QQ))$. To see this, we may assume that $A$ is a general complex torus. Then the summand $H^2(A,\QQ)\subset H^2(K^3(A),\QQ)$ (resp. $\pi_{*}^{\tau}(H^2(A,\QQ))\subset H^2(W_{\tau},\QQ)$)
		is the transcendental part of the cohomology of $K^3(A)$ (resp. of $W_{\tau}$). Since $\iota_{\tau}^*$ is a morphism of Hodge structures, we have our claim.
		
		By Remark \ref{rmk:BBclass}, we have $$
		\overline{q}_{W_{\tau}} = \frac{1}{4} \left(\sum_{i=1}^3 ({\lambda}^+_i)^2 - \sum_{i=1}^3 ({\lambda}^-_i)^2\right) -\frac{1}{2}\sum_{\alpha\in A_{2,\tau}} s_{\tau,\alpha}^2-\frac{1}{2} \delta_{\tau}^2.$$
		From Proposition \ref{prop:geometricInput} we know that $\iota^*_{\tau}(\xi) = 2\delta_{\tau}+\frac{1}{2}\sum_\alpha s_{\tau, \alpha}$. Expanding the square of this class leads to the claimed expression for $\iota_{\tau}^*\left(\overline{q}_{K^3(A)}\right) $. 
	\end{proof}
	
	\begin{lemma} \label{lem:restrictionW}
		Let $\tau\neq \tau' \in A_2$. In $H^4(W_{\tau},\QQ)$, we have:
		\begin{align*}
			\iota_{\tau}^*(w_{\tau'}) & = \frac{2}{5} \overline{q}_{W_{\tau}} + \frac{1}{4} \sum_{\alpha\in A_{2,\tau}} s_{{\tau}, \alpha}^2 - \frac{1}{4} \sum_{\alpha\in A_{2,\tau}} s_{{\tau}, \alpha}\cdot s_{{\tau}, \alpha+\tau+\tau'}.
		\end{align*}
	\end{lemma}
	\begin{proof}
		From Proposition \ref{prop:geometricInput1} we obtain the following intersection numbers:
		\begin{align*}
			\iota_{\tau}^*(w_{\tau'})\cdot \delta_{\tau}^2 & =\int_{V_{\tau, \tau'}} {\delta_{\tau}}_{|_{V_{\tau, \tau'}}}^2 = -4 ;\\
			\iota_{\tau}^*(w_{\tau'})\cdot \delta_{\tau}\cdot s_{\tau, \alpha} & =0 \ \ \text{ for any } \alpha\in A_{2,\tau}; \\
			\iota_{\tau}^*(w_{\tau'})\cdot s_{\tau,\alpha} \cdot s_{{\tau},\alpha+\theta} & =\begin{cases} -2 & \text{if } \theta\in \{0, \tau+\tau'\}; \\
				0 & \text{else}.
			\end{cases}
		\end{align*}
		Moreover, $\iota^*_{\tau}(w_{\tau'})\cdot \overline{q}_{W_{\tau}}= 30$.
		Indeed, since $\iota^*_{\tau}(w_{\tau'})$ is the fundamental class of $V_{\tau,\tau'}\subset W_{\tau}$ and $\overline{q}_{W_{\tau}}= \frac{5}{6} c_2(W_{\tau})$, the normal bundle sequence and Proposition \ref{prop:geometricInput1}.$(ii)$ yield
		\[
		\iota^*_{\tau}(w_{\tau'})\cdot \overline{q}_{W_{\tau}}= \frac{5}{6} \iota^*_{\tau}(w_{\tau'}) \cdot c_2(W_{\tau}) = \frac{5}{6} \left(\deg c_2(V_{\tau,\tau'})  + \deg c_2\bigl(N_{V_{\tau,\tau'}|W_{\tau}}\bigr) \right) = 30.
		\] 
		
		 By Remark \ref{rmk:changeBasis}, we can write~$\iota^*_{\tau}(w_{\tau'})$ as linear combination of the classes of Lemma \ref{lem:intersectionMatrix}:
		\[
		\iota_{\tau}^*(w_{\tau'}) = \eta_1 \overline{q}_{W_{\tau}} + \eta_2 \delta_{\tau}^2 + \eta_3 \delta_{\tau}\cdot\sum_{\alpha\in A_{2,\tau}} s_{{\tau},\alpha} + \sum_{\theta\in A_2} \beta_{\theta} \sum_{\alpha\in A_{2,\tau}} s_{{\tau}, \alpha} \cdot s_{{\tau}, \alpha+\theta} ,
		\]
		for suitable rational numbers $\eta_j$ and $\beta_{\theta}$, which we are going to determine using the intersection numbers computed above and the intersection matrix from Lemma \ref{lem:intersectionMatrix}.
		
			First, by the above, $\iota_{\tau}^*(w_{\tau'})$ is orthogonal to $\delta_{\tau}\cdot\sum_{\alpha}s_{\tau,\alpha}$, and hence $\eta_3=0$. Moreover, $\iota_\tau^*(w_{\tau'})$ is orthogonal to $\sum_{\alpha} s_{\tau,\alpha}\cdot s_{\tau,\alpha+\theta}$ unless $\theta=0$ or $\theta=\tau+\tau'$, and hence $\beta_{\theta}=0$ for~$\theta\notin\{0,\tau+\tau'\}$.
			For $\theta= \tau+\tau'$, we find
			\[
			\iota^*_{\tau}(w_{\tau'})\cdot \sum_{\alpha\in A_{2,\tau}}s_{\tau,\alpha}\cdot s_{\tau,\alpha+\tau+\tau'} = -32,
			\]
			which gives $128\beta_{\tau+\tau'} = -32$, i.e., $\beta_{\tau+\tau'}=-\frac{1}{4}$.
			Thus, 
			\[
			\iota^*_{\tau}(w_{\tau'}) = \eta_1 \overline{q}_{W_{\tau}} + \eta_2 \delta_{\tau}^2 +  \beta_{0} \sum_{\alpha\in A_{2,\tau}} s_{{\tau}, \alpha}^2 - \frac{1}{4} \sum_{\alpha\in A_{2,\tau}} s_{\tau, \alpha}\cdot s_{\tau, \alpha+\tau+\tau'}. 
			\]
			
			Computing now the intersection of $\iota^*_{\tau}(w_{\tau'})$ with the three classes $\overline{q}_{W_\tau}$, $\delta_{\tau}^2$ and $\sum_{\alpha} s_{{\tau},\alpha}^2$ via Lemma \ref{lem:intersectionMatrix} we obtain the equations
			\[
			\begin{cases}
				575\eta_1 - 50 \eta_2 - 800\beta_0 = 30;\\
				-50 \eta_1 + 12\eta_2 + 64\beta_0 = -4; \\
				-800 \eta_1 + 64\eta_2 + 1152\beta_0 = -32.
			\end{cases}
			\] 
			Solving this linear system we find $\eta_1=\tfrac{2}{5}$, $\eta_2=0$, $\beta_0=\tfrac{1}{4}$, as claimed.
	\end{proof}
	
	We finally calculate $\iota_{\tau}^*(w_{\tau})$. The strategy is similar to that used in the above proof. 
	\begin{proposition}\label{prop:c2normal}
		In $H^4(W_{\tau}, \QQ)$, we have
		\[
		\iota^*_{\tau} (w_{\tau}) = \frac{8}{5} \overline{q}_{W_{\tau}} + \delta_{\tau}^2 +  \sum_{\alpha\in A_{2,\tau}} s_{\tau, \alpha}^2 - \frac{1}{2}\delta_{\tau}\cdot \sum_{\alpha\in A_{2,\tau}} s_{\tau, \alpha}.
		\]
	\end{proposition} 
	\begin{proof} 
		We start by calculating the following intersection numbers:
		\begin{align*} \label{eq:someintersections}
			\iota_{\tau}^*(w_{\tau})\cdot \overline{q}_{W_\tau} & =70; \\ \iota^*_{\tau}(w_{\tau})\cdot \iota^*_{\tau}(w_{\tau'}) & = 12 \text{ for any } \tau \neq \tau'\in A_2; \\
			\iota^*_{\tau}(w_{\tau})\cdot \iota^*_{\tau}\left(\overline{q}_{K^3(A)}\right)& =84.
		\end{align*}
		Recall that $\iota^*_{\tau}(w_{\tau})$ equals the top Chern class $c_2(N_{W_{\tau}|K^3(A)})$ of the normal bundle. By~\S\ref{subsec:K3[2]}, we have $\overline{q}_{W_{\tau}}=\frac{5}{6} c_2(W_{\tau})$ and $\deg(c_4(W_{\tau})) = 324$. Via \eqref{eq:c2NW} we then find
		\[
		\overline{q}_{W_\tau} \cdot \iota_{\tau}^*(w_{\tau}) = \frac{5}{6} c_2(W_{\tau})\cdot c_2(N_{W_{\tau}|K^3(A)})=\frac{5}{6}\cdot \bigl(c_4(K^3(A))\cdot w_{\tau} - \deg(c_4(W_{\tau}))\bigr) = 70,
		\]
		as $c_4(K^3(A))\cdot w_{\tau}  = 408$ by Remark \ref{rmk:someValues}. For any $\tau\neq \tau'$ we have 
		\[
		\iota^*_{\tau}(w_{\tau})\cdot \iota^*_{\tau}(w_{\tau'}) = \int_{K^3(A)} w_{\tau}^2\cdot w_{\tau'} = 12
		\] 
		by Proposition \ref{prop:intersectionsW}. Finally,
		\[
		\iota^*_{\tau}(w_{\tau}) \cdot \iota_{\tau}^*\left(\overline{q}_{K^3(A)}\right) = \int_{K^3(A)} w^2_{\tau} \cdot \overline{q}_{K^3(A)} = C(\overline{q}_{K^3(A)} \cdot w^2_{\tau}) =84,
		\]
		where we used that $C(w^2_\tau)=12$ (by Remark \ref{rmk:someValues}) and \cite[Proposition 2.4]{BS22} to calculate the generalized Fujiki constant $C(\overline{q}_{K^3(A)}\cdot w^2_\tau)$.
		
		Now, according to Lemma \ref{lem:restrictCanonical} we can write 
		\[
		\iota^*_\tau (w_{\tau}) = \eta \cdot \iota_{\tau}^*\left(\overline{q}_{K^3(A)}\right) + \sum_{\theta\in A_2} \beta_{\theta} \cdot \sum_{\alpha\in A_{2,\tau}} s'_{\tau, \alpha}\cdot s'_{\tau, \alpha+\theta},
		\]
		for some coefficients $\eta$ and $\beta_{\theta}$, $\theta\in A_{2}$. 
		We claim that $\beta_{\theta}=\beta_{\theta'}$ whenever $\theta\neq 0, \theta'\neq 0$. Indeed,~$\iota^*_{\tau}(w_{\tau})\cdot (\iota^*_{\tau} (w_{\tau+\theta})-\iota^*_{\tau} (w_{\tau+\theta'}))=0$, and, by Lemma \ref{lem:restrictionW}, we have $$\iota^*_{\tau} (w_{\tau+\theta})-\iota^*_{\tau} (w_{\tau+\theta'})= -\frac{1}{4} \sum_{\alpha\in A_{2,\tau}} s_{\tau,\alpha} \cdot s_{\tau,\alpha+\theta}+ \frac{1}{4} \sum_{\alpha\in A_{2,\tau}} s_{\tau,\alpha} \cdot s_{\tau,\alpha+\theta'}.$$ 
		Using Remark \ref{rmk:changeBasis} and Lemma \ref{lem:intersectionMatrix}, we obtain the equation
		\[
		-16 \beta_{\theta} \cdot \frac{128}{4}  + 16 \beta_{\theta'} \cdot \frac{128}{4} =0,
		\] which proves the claim. We therefore have
		\begin{align*} 
			\iota^*_{\tau}(w_{\tau}) & = 
			\eta\cdot \iota^*_{\tau}\left(\overline{q}_{K^3(A)}\right) + \beta \sum_{\alpha\in A_{2,\tau}} (s'_{\tau,\alpha})^2 +\gamma \sum_{0\neq\theta\in A_{2}} \sum_{\alpha\in A_{2,\tau}} s'_{\tau,\alpha}\cdot s'_{\tau,\alpha+\theta},
		\end{align*} 
	for three coefficients $\eta,\beta,\gamma$. Via Remark \ref{rmk:changeBasis} and Lemma \ref{lem:restrictionBB}, we write $\iota_{\tau}^*(w_{\tau})$ in terms of the vectors of Lemma \ref{lem:intersectionMatrix} and we calculate the intersection of $\iota^*_{\tau}(w_{\tau})$ with the three classes $\overline{q}_{W_{\tau}}$, $\sum_{\tau'\neq\tau} \iota_{\tau}^*(w_{\tau'})$ and $\iota^*_{\tau}\bigl(\overline{q}_{K^3(A)}\bigr)$ in terms of the coefficients $\eta, \beta, \gamma$. 
	The computation is easily done (with the help of a computer) via Lemma \ref{lem:intersectionMatrix}, using the explicit formulae for $\iota_{\tau}^*(w_{\tau}')$ and $\iota_{\tau}^*\bigl(\overline{q}_{K^3(A)}\bigr)$ found in Lemma~\ref{lem:restrictionW} and Lemma~\ref{lem:restrictionBB} respectively. 
		We obtain the linear system
			\[
		\begin{cases}
			350\eta - 13600\beta - 12000\gamma & = 70; \\
			420\eta - 8640\beta - 22080\gamma & =180;\\
			252\eta - 7616\beta - 6720\gamma & = 84,
		\end{cases}
		\]
		which gives the solution $\eta = \tfrac{4}{5} $, $ \beta = \tfrac{9}{640} $, $\gamma = \tfrac{1}{640} $. Writing $\iota^*_{\tau}(w_{\tau})$ in terms of the vectors of Lemma \ref{lem:intersectionMatrix} leads to the claimed expression.		
	\end{proof}
	
	\section{Canonical Hodge classes of degree six} \label{sec:6}
	
	Our next goal is to show that canonical Hodge classes in the middle cohomology of a $\mathrm{Kum}^3$-manifold $K$ are analytic.
	We shall first determine the subalgebra of the cohomology which is fixed by the group $G$ (\S\ref{subsec:aut0}). We keep using the notation from the previous sections. 
	
	\begin{proposition}\label{prop:Ginvariants}
		In terms of the LLV-decomposition, we have 
		$$H^{\bullet}(K, \QQ)^G = V_{(3)} \oplus V_{(1,1)} \oplus 16V. $$
	\end{proposition}
	\begin{proof}
		We may and will assume that $K=K^3(A)$ is the generalized Kummer sixfold on an abelian surface $A$.
		Recall that the LLV-decomposition \eqref{eq:LLV} gives 	
		\[
		H^{\bullet}(K,\QQ) = V_{(3)}\oplus V_{(1,1)} \oplus 16V\oplus 240\QQ\oplus V_{\left(\tfrac{3}{2} , \tfrac{1}{2}, \tfrac{1}{2}, \tfrac{1}{2}\right)}.
		\]
		The group $G$ acts trivially on the summand $V_{(3)}$ and it has no invariants in the odd cohomology. Moreover, $G$ acts trivially on the summand $16V$, because, by Proposition~\ref{prop:canonicalH4}, this is generated as LLV-module by the $G$-invariant classes $c_2$ and $w_{\tau}$ for $\tau\in A_2$.		
		
		We prove that $V_{(1,1)}\subset H^{\bullet}(K,\QQ)^{G}$. Let $Y_{K}$ be the manifold of $\mathrm{K}3^{[3]}$-type given by Theorem \ref{thm:associatedK3}.
		By Remark \ref{rmk:iteratedBlowUp}, we have $Y_{K}=\tilde{K}/G$, where $\tilde{K}$ is the successive blow-up of~$K$ along $16$ smooth fourfolds $\bar{W}_\tau$. Each $\bar{W}_\tau$ is birational to the~$\mathrm{K}3^{[2]}$-variety $W_\tau$, and hence $h^{2,0}(\bar{W}_\tau)=1$. Since $h^{3,1}(K)=6$, the blow-up formula for cohomology yields the Hodge numbers $h^{4,0}(\tilde{K})=1$ and $h^{3,1}(\tilde{K}) = 6 + 16 = 22$.		
			On the other hand, we have~$h^{4,0}(Y_K)=1$ and $h^{3,1} (Y_{K})=22$, as for any manifold of $\mathrm{K}3^{[3]}$-type (\cite{GS1993}).
			By construction, we have an isomorphism $H^{\bullet}(Y_K,\QQ)=H^{\bullet}(\tilde{K},\QQ)^G$ and a $G$-equivariant embedding of $H^{\bullet}(K,\QQ)$ into~$H^{\bullet}(\tilde{K},\QQ)$. The Hodge numbers just calculated thus imply that the action of $G$ on $H^4(K,\CC)$ is trivial on $H^{4,0}(K)$ and $H^{3,1}(K)$. As Hodge structure, $V_{(1,1)}\cap H^4(K,\QQ)$ is a Tate twist of $H^2(K,\QQ)$; it must therefore consist of $G$-invariants, because $G$ commutes with the LLV-Lie algebra. We conclude that $V_{(1,1)} $ is contained in $H^{\bullet}(K,\QQ)^G$.   
			
			We next show that $G$ has no invariants in the LLV-summand $240\QQ$.
			To this end, we compute the traces of $g\in G= A_2\times \langle -1\rangle$ acting on this summand.
			Let $$\chi(g) \coloneqq \sum_i (-1)^i \mathrm{tr}\Bigl(g^*_{|_{H^i (K^3(A),\QQ)}}\Bigr).$$ 
			By the Lefschetz trace formula, $\chi(g)$ equals the Euler characteristic of the fixed locus~$K^g$, which consists of $8$ K3 surfaces for non-trivial $g=(\tau,1)$ and of the union of the~$\mathrm{K}3^{[2]}$-manifold $W_{\tau}$ and $140$ isolated points for $g=(\tau, -1)$ (see \cite[Proposition~2.10]{floccariKum3}). 
			Hence,
			\[
			\chi(g) = \begin{cases} 448 & \text{ if } g=1;\\
				464 & \text{ if } g=(\tau, -1) \text{ for } \tau\in A_2;\\
				192 & \text{ if } g=(\tau,1) \text{ for } 0\neq \tau\in A_2.
			\end{cases} 
			\]
			An automorphism $(\tau, \pm 1) \in G$ acts on the odd cohomology of $K$ as multiplication by $\pm 1$. 
			From the above and table \eqref{eq:ranks}, we deduce that
			\[
			\mathrm{tr}(g^*_{|240\QQ}) = \begin{cases} 240 & \text{ if } g=1;\\
				0  & \text{ if } g=(\tau, -1) \text{ for } \tau\in A_2;\\
				-16 & \text{ if } g=(\tau, 1) \text{ for } 0\neq \tau\in A_2.
			\end{cases} 
			\]
			We conclude that 
			\[
			\dim (240\QQ)^G = \frac{1}{|G|} \sum_{g\in G} \mathrm{tr}\left(g^*_{|240\QQ}\right)=\frac{1}{32} (240 - 16\cdot 15) = 0.
			\]
		\end{proof}
		\begin{remark}\label{rmk:characterizationG}
			Automorphisms in $\Aut_0(K)\setminus G$ act non-trivially on $H^4(K,\QQ)$, by Propositions \ref{prop:groupAction} and \ref{prop:canonicalH4}. Therefore, the above proposition characterizes $G$ as the group of automorphisms of $K$ which act trivially on the cohomology in degrees $2$ and $4$.
		\end{remark}
	
	\subsection{Canonical Hodge classes are analytic}
	As in the previous sections, we consider the generalized Kummer manifold $K^3(A)$ on a $2$-dimensional complex torus~$A$, and the submanifolds $\iota_{\tau}\colon W_{\tau}\hookrightarrow K^3(A)$, for $\tau\in A_2$. Recall that we introduced $16$  divisor classes $s'_{\tau,\alpha}$ on each $W_{\tau}$ in Definition~\ref{def:s'}.
	
	\begin{definition} \label{def:canonical3cycles}
		For any $\tau\in A_2$ and $\alpha\in A_{2,\tau}$ we define
		\[
		d_{\tau, \alpha}\coloneqq \iota_{\tau,*}(s'_{\tau, \alpha}) \in H^6(K^3(A),\ZZ).
		\]
	\end{definition} 
	
	\begin{lemma} \label{lem:deformD}
		The classes $d_{\tau, \alpha}$ are canonical Hodge classes, and they remain analytic on any deformation of $K^3(A)$.
	\end{lemma} 
	\begin{proof}
		Let $\mathcal{K}\to B$ be a deformation of $K^3(A)$, and consider the induced deformation~$\mathcal{W}_{\tau}\to B$ of $W_{\tau}$ as in Remark \ref{rmk:deformations}. A class $d_b\in H^6(\mathcal{K}_b,\QQ)$ obtained as parallel transport of some $d_{\tau,\alpha}$ is of the form $\iota_{\tau,b,*}(s'_b)$, where $s'_b\in H^2(\mathcal{W}_{\tau,b},\ZZ)$ is the parallel transport of $s'_{\tau,\alpha}$; here, $\iota_{\tau,b}\colon \mathcal{W}_{\tau,b} \hookrightarrow \mathcal{K}_b$ is the inclusion.
		By Lemma \ref{lem:Hodge}, $s'_b$ is a Hodge class. By Lefschetz (1,1) theorem we have $s'_b=c_1(L_b)$ for some line bundle $L_b$ on $\mathcal{W}_{\tau, b}$. A Grothendieck-Riemann-Roch computation (using that hyper-K\"ahler manifolds have trivial first Chern class) shows that $d_b=\frac{1}{2}c_3\left(\iota_{\tau,b,*}(L_b)\right)$; hence, $d_b$ is an analytic class on~$\mathcal{K}_b$. 
	\end{proof}
	
	The next statement thus implies that canonical Hodge classes of degree~$6$ are analytic, for any manifold of $\mathrm{Kum}^3$-type.
	\begin{proposition}\label{prop:canonicalH6}
		The classes $d_{\tau, \alpha}$ generate the subspace of canonical Hodge classes in $H^6(K^3(A), \QQ)$.
	\end{proposition}
	
	We will use the following intersection numbers to prove the above proposition.	
	\begin{proposition}\label{prop:intersectionsD}
		Fix $\tau\in A_{2}$. For $\alpha, \alpha'\in A_{2,\tau}$, we have $$d_{\tau, \alpha}\cdot d_{\tau, \alpha'} = \begin{cases} -52 & \text{ if } \alpha= \alpha';\\
			12 & \text{ if } \alpha\neq \alpha'.
		\end{cases} $$
	\end{proposition}
	\begin{proof}
		Since $c_2\left(N_{W_{\tau}|K^3(A)}\right) = \iota^*_{\tau}(w_{\tau})$, by \cite[Corollary 6.3]{fulton} we have 
		\[
		d_{\tau, \alpha}\cdot d_{\tau, \alpha'} = \int_{W_{\tau}} \iota^*_{\tau}(w_{\tau})\cdot s'_{\tau, \alpha} \cdot s'_{\tau, \alpha'},
		\]
		for any $\alpha,\alpha'\in A_{2,\tau}$. 
		It is straightforward to compute these numbers with the intersection form~\eqref{eq:intersectionForm}, using the expression for $\iota_{\tau}^*(w_{\tau})$ of Proposition \ref{prop:c2normal}, the definition of $s'_{\tau,\alpha}$ (Definition \ref{def:s'}) and the intersections determined in the proof of Lemma \ref{lem:intersectionMatrix}.
	\end{proof}
	
	\begin{proof}[{Proof of Proposition \ref{prop:canonicalH6}}]
		Fix $\tau\in A_2$. With the intersection numbers of Proposition~\ref{prop:intersectionsD}, an argument entirely analogous to that used in the proof of Proposition \ref{prop:canonicalH4} shows that the~$16$ classes $d_{\tau, \alpha}$, for $\alpha\in A_{2, \tau}$, are linearly independent. 
		The $G$-invariants in $\langle d_{\tau, \alpha}\rangle_{\alpha\in A_{2,\tau}}$ are the multiples of~$\sum_{\alpha} d_{\tau,\alpha}$ (cf. Remark \ref{rmk:actionDivisors}). This class is non-zero since~$(\sum_{\alpha} d_{\tau,\alpha}) \cdot (\sum_{\alpha} d_{\tau,\alpha}) = 2048$ by Proposition~\ref{prop:intersectionsD}.
		But, by Propositions \ref{lem:Gammainvariants} and \ref{prop:Ginvariants}, any $G$-invariant canonical Hodge class in $H^6(K^3(A),\QQ)$ is actually invariant under the action of the entire group $\Aut_0(K^3(A))$. Since translation by $\epsilon\in A_4$ maps~$d_{\tau,\alpha}$ to~$d_{\tau+2\epsilon, \alpha+\epsilon}$, we must have~$\sum_{\alpha\in A_{2,\tau}} d_{\tau, \alpha} = \sum_{\alpha'\in A_{2,\tau'}} d_{\tau',\alpha'}$, for any $\tau, \tau'\in A_2$. 
		
		Next, we observe that for $\tau\neq \tau'$ in $A_2$ the value $d_{\tau, \alpha}\cdot d_{\tau', \alpha'}$ is independent of $\alpha$ and $\alpha'$. 
		Indeed, by Definitions \ref{def:s'} and \ref{def:canonical3cycles}, 
		and Remark \ref{rmk:mixedIntersections}, the number in question equals 
		$$-4 (\iota_{\tau, *}(s_{\tau, \alpha}) \cdot \iota_{\tau',*}(\delta_{\tau'})+  \iota_{\tau,*}(\delta_{\tau})\cdot  \iota_{\tau', *}(s_{\tau', \alpha'})),$$ and using the symmetries in $G$ we see that the numbers $\iota_{\tau, *}(s_{\tau, \alpha})\cdot \iota_{\tau',*}(\delta_{\tau'})$ do not depend on~$\alpha\in A_{2,\tau}$.
		Since we have just shown that $\sum_{\alpha\in A_{2,\tau}} d_{\tau, \alpha} = \sum_{\alpha'\in A_{2,\tau'}} d_{\tau',\alpha'}$, Proposition~\ref{prop:intersectionsD} gives 
		\[
		\left(\sum_{\alpha\in A_{2,\tau}} d_{\tau, \alpha} \right)\cdot d_{\tau', \alpha'} =128,
		\]
		Therefore $d_{\tau, \alpha}\cdot d_{\tau' , \alpha'} = 8$ whenever $\tau\neq \tau'$.
		
		Assume now that a linear combination $u= \sum_{\tau\in A_2} \sum_{\alpha\in A_{2,\tau}} \beta_{\tau,\alpha}\cdot d_{\tau,\alpha}$ is zero. Taking the intersection product of $u$ with $(d_{\tau, \alpha} - d_{\tau, \alpha'})$ we obtain that $\beta_{\tau,\alpha}=\beta_{\tau, \alpha'}$ for all $\alpha, \alpha'\in A_{2,\tau}$. Thus we can write 
		$u= \sum_{\tau \in A_2} \gamma_{\tau} \sum_{\alpha\in A_{2,\tau}} d_{\tau,\alpha}$, and hence we must have $\sum_{\tau} \gamma_{\tau}=0$.
		This shows that the $\QQ$-vector space of linear relations between the $256$ classes $d_{\tau, \alpha}$ is generated by $\sum_{\alpha} d_{\tau, \alpha} - \sum_{\alpha'} d_{\tau' , \alpha'}, $ for $\tau, \tau'$ in $A_2$. This space has dimension $15$, and hence the $256$ classes $d_{\tau, \alpha}$ span a subspace of $H^6(K^3(A),\QQ)$ of dimension $241$, which therefore must be the whole $241$-dimensional space of canonical Hodge classes in $H^6(K^3(A),\QQ)$.
	\end{proof}

\section{Algebraic cycles on $\mathrm{Kum}^3$-varieties}\label{sec:7}

The work done so far proves the following.
\begin{theorem}\label{thm:HCcanonical}
	Let $K$ be any manifold of $\mathrm{Kum}^3$-type. All canonical Hodge classes on~$K$ are analytic.
\end{theorem} 
\begin{proof}
	The assertion for canonical Hodge classes in $H^4(K,\QQ)$ (resp. in $H^6(K,\QQ)$) follows from Lemma \ref{lem:deformingW} and Proposition \ref{prop:canonicalH4} (resp. from Lemma \ref{lem:deformD} and Proposition~\ref{prop:canonicalH6}). By Remark \ref{rmk:productq}, the cup-product with $\overline{q}_K$ induces an isomorphism between canonical Hodge classes in $H^4(K,\QQ)$ and those in $H^8(K,\QQ)$, which therefore are analytic as well.
\end{proof}

We can now complete the proof of our main result.

\begin{proof}[Proof of Theorem \ref{thm:HCKum3}] 

Let $K$ be any projective manifold of $\mathrm{Kum}^3$-type, and consider the action of the group $G\subset \Aut_0(K)$ on the cohomology. By Proposition \ref{prop:Ginvariants}, we can write:
\[
H^{\bullet}(K,\QQ) = (H^{\bullet}(K,\QQ))^G \oplus H^{\mathrm{odd}}(K,\QQ) \oplus 240\QQ.
\]
The odd cohomology obviously does not contain Hodge classes. The last summand consists of canonical Hodge classes, which are all algebraic by Theorem \ref{thm:HCcanonical}. It remains to show that the Hodge classes in $(H^{\bullet}(K,\QQ))^G$ are algebraic. 

By Theorem \ref{thm:associatedK3}, we have a diagram
\[
\begin{tikzcd}
	& \tilde{K} \arrow{dr}{q} \arrow[swap]{dl}{b} \\
	K \arrow[dashed]{rr}{r} && Y_K,
\end{tikzcd}
\]
where $b\colon \tilde{K}\to K$ is the successive blow-up of $16$ smooth fourfolds, the action of $G$ extends to an action on $\tilde{K}$, and the quotient $q\colon \tilde{K}\to Y_K$ gives a variety of $\mathrm{K}3^{[3]}$-type. 
Moreover, $Y_K$ is birational to a smooth and projective moduli space $M$ of stable sheaves on the K3 surface $S_K$. The Hodge conjecture holds for any power of $S_K$ by \cite[Corollary 5.8]{floccariKum3}; by a result of B\"ulles~\cite{Bue18}, this implies that the Hodge conjecture holds for $M$ as well. Since birational hyper-K\"ahler varieties have isomorphic motives (\cite{Rie14}), it follows that the Hodge conjecture holds for $Y_K$.

The pull-back along $q$ identifies the cohomology of $Y_K$ with $H^{\bullet}(\tilde{K},\QQ)^G$, while the pull-back along $b$ embeds $G$-equivariantly the rational cohomology of $K$ into that of $\tilde{K}$. The left-inverse of $b^*$ is the push-forward $b_*$. Therefore, $r^*\coloneqq b_*q^*$ induces a surjection
$$r^*\colon H^{\bullet}(Y_K,\QQ) \to (H^{\bullet}(K,\QQ))^G.$$
Note that $r^*$ is induced by an algebraic correspondence. Since the Hodge conjecture holds for $Y_K$, any Hodge class in $(H^{\bullet}(K,\QQ))^G$ is algebraic. 
\end{proof}

\begin{proof}[Proof of Theorem \ref{thm:TCKum3}]
	By the works \cite{floccari2019}, \cite{soldatenkov19} and \cite{FFZ}, the Mumford-Tate conjecture holds for any hyper-K\"ahler variety of known deformation type; the final result may be found in \cite[Theorem 1.18]{FFZ}.
	See \cite[\S2.1]{moonen2017} for the statement of this conjecture. If $X/k$ is a smooth and projective variety over the finitely generated field $k\subset \CC$ for which the Mumford-Tate conjecture holds, then the Galois representation on the $\ell$-adic cohomology of $X$ is semisimple and the Hodge conjecture for $X_{\CC}$ is equivalent to the Tate conjecture for $X/k$. Therefore, Theorem \ref{thm:TCKum3} follows from Theorem \ref{thm:HCKum3}.
\end{proof}

\bibliographystyle{smfplain} 
\bibliography{bibliographyNoURL}{}

\end{document}